\documentclass[11pt]{amsart}

\usepackage{verbatim}

\usepackage{tikz-cd}

\usepackage[backref,bookmarks=false]{hyperref}

\usepackage{color}
\newcommand{\blue}{\color{blue}}

\numberwithin{equation}{section}

\newcommand{\ben}{\begin{enumerate}}
\newcommand{\een}{\end{enumerate}}

\newcommand{\bea}{\begin{eqnarray}}
\newcommand{\ba}{\begin{array}}
\newcommand{\bean}{\begin{eqnarray*}}
\newcommand{\ea}{\end{array}}
\newcommand{\eea}{\end{eqnarray}}
\newcommand{\eean}{\end{eqnarray*}}
\newcommand{\beq}{\begin{equation}}
\newcommand{\eeq}{\end{equation}}
\newcommand{\bthm}{\begin{thm}}
\newcommand{\ethm}{\end{thm}}
\newcommand{\blem}{\begin{lem}}
\newcommand{\elem}{\end{lem}}
\newcommand{\bprop}{\begin{prop}}
\newcommand{\eprop}{\end{prop}}
\newcommand{\bcor}{\begin{cor}}
\newcommand{\ecor}{\end{cor}}
\newcommand{\bdfn}{\begin{dfn}}
\newcommand{\edfn}{\end{dfn}}
\newcommand{\brem}{\begin{rem}}
\newcommand{\erem}{\end{rem}}
\newcommand{\bpf}{\begin{proof}}
\newcommand{\epf}{\end{proof}}
\newcommand{\bfact}{\begin{fact}}
\newcommand{\efact}{\end{fact}}
\newcommand{\bobs}{\begin{obs}}
\newcommand{\eobs}{\end{obs}}
\newcommand{\bexam}{\begin{exam}}
\newcommand{\eexam}{\end{exam}}
\newcommand{\bclaim}{\begin{claim}}
\newcommand{\eclaim}{\end{claim}}

\newtheorem{thm}{Theorem}[section]
\newtheorem{prop}[thm]{Proposition}
\newtheorem{lem}[thm]{Lemma}

\newtheorem{cor}[thm]{Corollary}
\newtheorem{dfn}[thm]{Definition}
\newtheorem{rem}[thm]{Remark}
\newtheorem{fact}[thm]{Fact}
\newtheorem{claim}[thm]{Claim}
\newtheorem{obs}[thm]{Observation}
\newtheorem{exam}[thm]{Example}

\newtheorem*{condition'}{Condition 2'}

 \newtheoremstyle{claimstyle}%
   {}
   {}
   {\normalfont}
   {}
   {\itshape}
   {.}
   { }
   {\thmnote{#3}}

\theoremstyle{claimstyle}

\alph{enumii} \roman{enumiii}

             \def\cB{\mathcal B}       \def\cC{\mathcal C}
\def\cH{\mathcal H}                    
\def\cL{{\mathcal L}}                   
\def\cU{\mathcal U}             \def\cV{\mathcal V}          \newcommand{\J}{\mathcal{J}}
                           
\def\cW{\mathcal W}           \def\cK{\mathcal K}         
 \def\cI{\mathcal I} \def\cD{\mathcal D}

                \def\Z{{\mathbb Z}}      \def\R{{\mathbb R}}
\def\C{{\mathbb C}}

                \def\b{\beta}             \def\d{\delta}
\def\De{\Delta}                          
                \def\Ga{\Gamma}           \def\l{\lambda} 
              \def\om{\omega}           \def\Om{\Omega}
               \def\sg{\sigma}
\def\Th{\Theta}                          
\def\ka{\kappa}

\newcommand{\ep}{\varepsilon}
\newcommand{\ph}{\varphi}
\newcommand{\al}{\alpha}
\newcommand{\ga}{\gamma}

\def\1{1\!\!1}

\def\and{\text{ and }}

\def\dist{\text{{\rm dist}}}

\def\h{{\text h}}

\def\Int{\text{{\rm Int}}}
         \def\P{\text{{\rm P}}}

\def\({\bigl(}                \def\){\bigr)}

                        \def\^{\tilde}

            \def\sms{\setminus}
\def\sbt{\subset}

\def\ov{\overline}

\def\om{\omega}

\def\D{{\mathbb D}}

\def\${$ \displaystyle }

\newcommand{\pft}{{\mathcal{L}}_t}
\newcommand{\pftf}{{\mathcal{L}}_{t,f}}
\newcommand{\pftF}{{\mathcal{L}}_{t,F}}
\newcommand{\wh}{{\mathbf{W}}}

\def\disf{{\bf f}}
\def\disF{{\bf F}}

\newcommand{\pf}{{\mathcal{L}}}

\newcommand{\jul}{\mathcal J}
\newcommand{\fat}{\mathcal F}

\def\Tract{\Omega}

\def\lm{{\rm l}_m}
\def\lmm{{\rm l}}

\begin{document}

\title[The failure of Ruelle's property for entire functions]{The failure of Ruelle's property for entire functions}


\author[\sc Volker MAYER]{\sc Volker Mayer}
\address{Volker Mayer, Universit\'e de Lille, D\'epartement de
  Math\'ematiques, UMR 8524 du CNRS, 59655 Villeneuve d'Ascq Cedex,
  France} \email{volker.mayer@univ-lille.fr \newline
  \hspace*{0.42cm} \it Web: \rm math.univ-lille1.fr/$\sim$mayer}

  \author[\sc Anna ZDUNIK]{\sc Anna Zdunik}
\address{Anna Zdunik, Institute of Mathematics, University of Warsaw,
ul. Banacha 2, 02-097 Warszawa, Poland}
\email{A.Zdunik\@@mimuw.edu.pl}

\date{\today} \subjclass{Primary 37F10; Secondary 30D05, 37F45, 28A80.}
\thanks{This work was supported in part by the National Science Centre, Poland, grant no 2018/31/B/ST1/02495 and
by the Labex CEMPI  (ANR-11-LABX-0007-01).}


\begin{abstract}  
We exhibit an analytic family of hyperbolic, even disjoint type, entire functions 
for which the hyperbolic dimension does not vary analytically. Additionally we answer several
questions in thermodynamic formalism of entire functions such as the existence of 
a hyperbolic entire function without conformal measure  that is supported
on the radial Julia set.
 \end{abstract}

\maketitle


\vspace{-0.3cm}


\section{Introduction} \label{s1}

Ruelle \cite{Ru82}, answering a conjecture of Sullivan, has shown  that the Hausdorff dimension 
of the Julia set of hyperbolic rational functions depends analytically on the map. An alternative approach
to this result is contained in the monograph \cite{Zins2000} and, following Bishop \cite{Bishop06},
we will call it Ruelle's property.
The hyperbolicity assumption is essential in this result (see \cite{Shi98} and \cite{DSZ97}) and thus,
 in all what follows, we assume that the analytic families under considerations always have this property.

The paper \cite{Ru82} has been published in 1982 and since then this property has been generalized in many ways. 
Ruelle himself also established it for 
analytic quasiconformal deformations of cocompact Fuchsian groups,
a result which has been extended by Anderson and Rocha \cite{AR97} to convex co-compact 
Kleinain groups. There is also a version for Henon maps in $\C^2$ by Verjovsky and Wu \cite{VW96}, for rational
semi-groups by Sumi and Urba\'nski \cite{Sumi09} and one for hyperbolic surface diffeomorphisms 
by Pollicott \cite{Po15}. Employing Birkhoff's cone method,  Rugh \cite{Rugh} extended recently Ruelle's property to random $C^1$--conformal repellers.

The common tool in all analyticity results is Bowen's formula (see \cite{Bow79} for the original version) which expresses the dimension in terms of the zero of a 
pressure function. One should have in mind that this formula really determines the \emph{hyperbolic dimension}
which is  the supremum over the Hausdorff dimensions of hyperbolic subsets of the Julia set (see Shishikura \cite{Shi98}).
For most  rational functions, in particular for all hyperbolic ones, the hyperbolic dimension coincides with the Hausdorff dimension of the Julia set.
In transcendental dynamics however the situation is different: in general, there is a definite gap between these two dimensions 
(see \cite{St99-1} and \cite{UZ03}) and a typical phenomenon in the case of entire functions is that the dimension of the Julia set itself is often maximal, i.e. equal to  2. The later was first observed by McMullen \cite{McM87} and Bara\'nski \cite{B08} has shown that this property also holds for all entire functions of finite order and of class $\cB$, 
the class  introduced by Eremenko and Lyubich \cite{EL92} which consists in the entire functions that have a bounded singular set
(see Section 2.1 for the definitions of the singularities).
The intriguing thing then is how the hyperbolic dimension behaves.


Urba\'nski and Zdunik showed analytic variation of the hyperbolic dimension at hyperbolic parameters $\l\in \C^*$ for the exponential family $\l e^z$ in \cite{UZ04}.
After this first result for transcendental dynamics, many contributions where made. The papers \cite{UZ04, MyUrb08}
along with Skorulski and Urba\'nski's results in \cite{SU14} show that Ruelle's property does hold in great generality and for most of the 
classical families of transcendental functions, in particular for
 the sine, the tangent and the Weierstrass elliptic family. It also has been established in the realm of random dynamics for a class of transcendental functions 
in \cite{Mayer:2016aa} and even beyond the scope of hyperbolic functions. Indeed,  Kotus and Urba\'nski  \cite{KoUrb} considered a family of Fatou's function that have a persistent Baker domain and for which the hyperbolic dimension still behaves analytically.

 Given all these results, real analytic dependence of the hyperbolic dimension does hold in great generality in transcendental dynamics. Contrary to that, we provide here the first example
of an analytic family of entire hyperbolic functions for which Ruelle's property breaks down.

\bthm\label{theo main}
There exists a holomorphic family of finite order entire functions $\disF_\l= \l \, \disF$, $  \l\in \C\setminus \{0\}$, 
of class $\cB$
such that 
the functions $\disF_\l$, $\l\in (0, 1]$, 
are all in the same hyperbolic component of the parameter space but the function
$$\l \mapsto HypDim(\disF_\l )$$
is not  analytic in $(0, 1]$,  where  $HypDim(\disF _\l )$ denotes the hyperbolic dimension of  $\disF_\l$.
\ethm

For limit sets of Kleinian and Fuchsian groups such a break down of  Ruelle's property was observed
initially by
Astala and Zinsmeister \cite{AstZins94}.
They gave an example of an analytic family of infinitely generated quasifuchsian groups 
for which Ruelle's analyticity result does not hold.
Bishop \cite{Bishop06} subsequently extended their result and gave a criterion for the failure
of analyticity for a class of infinitely generated quasifuchsian groups.
More recently, Huo and Wu \cite{HuoWu15} established an analogous result for deformations of Fuchsian groups of the second kind.

\medskip


Functions of class $\cB$ have only logarithmic singularities over infinity (see Section 2.1)
and the functions $\disF_\l$ of Theorem \ref{theo main}
are built in such a way that they
have only one logarithmic singularity over infinity but  a very special one. For such functions we dispose in a complete
theory of thermodynamic formalism  \cite{MUpreprint3}. This theory relies on the behavior of the transfer operator (see Section \ref{s3} $\,$
for the definition and properties of the transfer operator) and
it is shown in \cite{MUpreprint3} that there exists a transition parameter $\Th \geq 1$ such that the series giving the transfer operator
 with parameter $t$ is divergent if $t<\Th$ and convergent, even a bounded operator, if $t>\Th$. Moreover,  it allows us to get   precise estimates for the transfer operator and of the transition parameter
in terms of the fractal geometry 
of the singularity over infinity. 
Using them, we are able to construct entire functions for which the transfer operator at its transition parameter $t=\Th$ is convergent. We do this in fact  by constructing first a model function (Sections 3-6) and then (Sections 7-10) carry all the properties over to an entire function using 
the approximation method of Rempe in \cite{Rempe-HypDim2}.

\medskip

It turns out that our approach also answers several other open questions.
The first result  answers positively the question in Remark 3.7 in \cite{BKZ18}
(see Section \ref{s9} for the precise definitions of the notions in the following results such as topological pressure and conformal measure).


\bthm\label{theo main''}  For every $1<\Th<2$ there exists a disjoint type and finite order entire function  $ f\in\cB$
whose transfer operator has transition parameter $\Th$, such that the transfer operator is convergent at $\Th$
and such that the topological pressure at $t=\Th$ is  strictly negative. Consequently, the topological pressure 
of $f$ has no zero. 
\ethm

We also can complete the picture concerning the behavior of the hyperbolic dimension.
For an entire function $f$ having a tract of sufficiently nice geometry it is known that $HypDim(f)\geq \Th\geq 1$
where this time $\Th$ is a  transition parameter of $f$ restricted to this tract (see \cite{Mayer:2017ab}).
Moreover,  when $\Th =1$ then $HypDim(f)>1$ (this strict inequality has previously been obtained in full generality in \cite{BKZ-2009}). The functions in the present paper show that strict inequality between the hyperbolic dimension and the transition parameter is no longer true as soon as 
 $1<\Th <2$. The case $\Th=2$ was studied by Rempe-Gillen in \cite{Rempe-HypDim2} where a disjoint type function of finite order and with  hyperbolic dimension equal to two  was constructed.

\bthm\label{theo main'''} For every $1<\Th<2$ 
there exists a disjoint type  and finite order  entire function $ f\in\cB$ with a single quasidisk tract and whose hyperbolic dimension attains the minimal possible value $HypDim(f)=\Th$.
\ethm

Finally, the functions of Theorem \ref{theo main'''} also explain that hyperbolic, even disjoint type, entire functions can behave like the very flexible, since locally defined, irregular conformal iterated function systems (see \cite{MauldinUrb02}).

\bthm\label{theo main'}  For every $1<\Th<2$ there exists a disjoint type  and finite order entire function $f\in \cB$ 
such that $HypDim(f)=\Th $ and such that $f$ does not have a conformal measure
supported on its radial Julia set.
\ethm

\noindent
 {\em Acknowledgement:} The authors thank the anonymous referee whose remarks, comments, and suggestions allowed us to improve the final exposition of the paper.

\section{Preliminaries}
Let
$\D(z,r)$ be the open disk with center $z\in \C$ and radius $r>0$.
When the center is the origin, we also use the notation
$
\D_r:=\D(0,r)
$
 and then the complement of its closure will be denoted by
$$
\D_r^*=\C\sms \ov\D_r.
$$
 We also consider the half-spaces
$$ \cH_s=\big\{z\in \C \; , \;\; \Re {z} >s\big\} \quad , \quad s\geq 0 \, .$$
 When $s=0$, then we also write $\cH$ for $\cH_0$.

\subsection{Order and singularities}
Let  $f:\C\to\C$ be  an entire function.
 The order $\rho (f)$ of $f$ is defined by
\beq \label{order}
\rho (f) =\limsup_{|z|\to\infty} \frac{\log \log |f(z)|}{\log |z|}
\eeq
and $f$ is called of \emph{finite order} if $\rho (f) <\infty$.

 Iversen's classification of singularities is very well explained in \cite{BE08.1}.
An entire function $f$ can have two types of singular values: $b\in \C$ is a \emph{critical value} if $b=f(c)$ for some $c\in \C$ with $f'(c)=0$ and 
and $b\in \hat \C$ is an \emph{asymptotical value} if there exists a curve $\ga \subset \C$ tending to infinity and such that $f(z)\to b$
as $z\to\infty$, $z\in \ga$. In this case there exists for every $r>0$ an unbounded connected component $\Om_r$ of $f^{-1}(\D(b,r))$
such that $\Om_{r'}\subset \Om _r$ if $r'<r$ and $\bigcap_{r>0} \Om_r =\emptyset$. Such a choice of components is called singularity over $b$
and it is called \emph{logarithmic singularity} in the particular case when $f:\Om_r\to \D(b,r)\setminus \{b\}$ is a universal covering for some $r>0$.
The set of critical values and of finite asymptotic values of $f$
 will be denoted by $S(f)$.

We consider functions of the Eremenko--Lyubich class $\cB$ that consists of entire functions for which $S(f)$ is a bounded set. These functions are also called of \emph{bounded type. 
}
If $f\in \cB$, then there exists $r>0$ such that $
S(f)\subset \D_r
$.
Then $f^{-1}(\D_r^*)$ consists of mutually disjoint unbounded Jordan domains $\Omega$ with real analytic boundaries such that $f:\Om\to \D^*_r$ is a covering map (see \cite{ EL92}). Thus, an entire function $f$ of class $\cB$
has only  logarithmic singularities over infinity. The connected components of $f^{-1}(\D^*_r)$ are called \emph{tracts} or, more precisely, \emph{logarithmic tracts} and 
the restriction of $f$ to any of these tracts $\Om$ has the special form 
\beq\label{Nov20 1}
\text{$f_{|\Om} =\exp\circ{\tau}\;\;$
where $\;\;\ph=\tau^{-1} :\cH_{\log r} \to \Om$}
\eeq
is a conformal map fixing infinity. In the following we  use for every $s\geq 0$ the notation
 $\Om_s= \ph (\cH_s )$
 so that, in particular,
 $$\Om=\Om_{\log r}=\ph (\cH_{\log r}) \quad \text{and}\quad \Om_0=\ph (\cH _0)=\ph (\cH).$$

 \noindent
In this work we construct entire functions of class $\cB$ having just one particular tract $\Om$.


\subsection{Model functions} 

Besides globally defined entire functions we also consider holomorphic functions that are only defined in
a unbounded simply connected domain $\Om$ where it has the form \eqref{Nov20 1}. Such functions are called
models and the following is a simple version of the general definition,
see Bishop \cite{Bishop-EL-2015, Bishop-S-2016}.

\bdfn\label{tract model}
A model is a holomorphic map $$f=e^\tau :\Om \to \D_r^*= \{ |z| >r\}$$
where $\Om$ is a simply connected unbounded domain, called tract,
such that $\partial\Om$ is a connected subset of $\C$ where $r\geq 1$ and where
$\tau : \Tract \to \cH_{\log r}$ is a conformal map fixing infinity:
$$\text{$\tau (z)\to \infty\;\;$ if $\;\;z\to \infty$.}$$
\edfn

 In general, a model function can clearly not be extended to an entire function but it can be approximated
by entire functions in various ways (see \cite{Bishop-EL-2015, Bishop-S-2016, Rempe-HypDim2}).
Since we use intensively  \cite{Rempe-HypDim2} we provide all necessary properties of this approximation in  Section \ref{s7}.

\smallskip

\subsection{Dynamical preliminaries}
All relevant informations on the dynamics of transcendental functions can be found in Bergweiler's
 survey \cite{Bergweiler-survey}. As for rational functions,
\emph{the Fatou} set $\fat _f$ of an entire function $f$ is the set of points of the plane that admit a neighborhood 
on which the iterates $f^n$, $n\geq 1$, are normal with respect to the spherical metric.
The complement 
$\jul _f= \C\setminus \fat _f$ is the \emph{Julia set} of $f$.

An entire function $f:\C\to\C$ is called \emph{hyperbolic} if  there is a compact set $K$ such that 
$$
f(K)\subset \Int(K)
$$ 
and $f:f^{-1}(\C\setminus K)\to \C\setminus K$ is a covering map. 
This extends naturally the notion of hyperbolicity of the rational to the transcendental case since, according
to Theorem 1.3 in 
\cite{RempeSixsmith16}, an entire function $f$ is hyperbolic if and only if the postsingular set 
$$
\P(f):= \overline { \bigcup_{n\geq 0} f^n (S(f))}
$$
is a compact subset of the Fatou set of $f$.
Clearly, a hyperbolic function belongs to class $\cB$.

\emph{Disjoint type} functions are particular hyperbolic functions. This notion first implicitly appeared in \cite{B07} and means that
the compact set $K$ in the definition of a hyperbolic function can be taken to be connected. In this case, the Fatou set of $f$ is connected. For example, if $f\in \cB$ and if there exists $\cD$ a simply connected bounded domain such that 
\beq\label{old.1}
S(f)\subset \cD \quad  \  \text{and} \quad  f^{-1}(\C\setminus \ov\cD) \cap \ov \cD=\emptyset\ 
\eeq
then $f$ is of disjoint type. A particular case for the domain $\cD$ is a disk centered at the origin.

\smallskip

A model $f :\Om \to \D_r^*$ is said to be of disjoint type if $\overline \Om \subset \D_r^*$ and, in this case, the Julia set 
is $$\J_f =\{ z\in \Om \; ; \;\; f^n(z)\in \Om \; \text{for all}\; n\geq 1\}.$$
This is consistent with the above definition of the Julia set for disjoint type entire functions.

\smallskip

Concerning the radial Julia set, there are several definitions in the literature (see \cite{MUmemoirs, Rempe09-hyp}).
It is explained in Remark 4.1 of \cite{MUmemoirs} that these definitions lead to different sets whose difference  is dynamically insignificant. In particular they 
have same Hausdorff dimension. Since we deal only with hyperbolic, in fact disjoint type, entire or model functions, the following definition
fits best to our context:
$$\J_r(f) = \{z\in \J(f) \;: \; \liminf_{n\to\infty} | f^n(z) | <\infty\}\,$$
and clearly this is a Borel set.
According to \cite{Rempe09-hyp}, the Hausdorff dimension of this set equals the \emph{hyperbolic dimension} of $f$. 
$$HypDim(f)= Hdim ( \J_r (f))\,.$$
Falconer's book \cite{Falconer} contains all relevant informations on fractal dimensions. One can find the definition of Hausdorff dimension in Section 3 and the one of Minkowski dimension in Section 2.

\smallskip

We will consider an analytic family of maps of the form
$f_\l = \l \, f$, $\l \in \C^*,$
where $f\in \cB$ is a given entire function.

\bdfn\label{dfn component} The functions $f_{\l_1}, f_{\l_2}$ belong to the same hyperbolic component of $\C^*$ if there exists
a simply connected domain $V\subset \C^*$ that contains  $\l_1,\l_2$ and such that
\ben
\item all the functions $f_\l$, $\l\in V$, are hyperbolic and 
\item the functions $f_\l$, $\l\in V$, are $J$--stable in the sense of holomorphic motions:  there exists a base point 
$\l_0\in V$ and a holomorphic motion $(\ph_\l)_{\l\in V}$ identifying the Julia sets $\ph_\l (\J_{\l_0})=\J_\l$ and conjugating the dynamics on the Julia sets, i.e.
for  $\l\in V$ we have $\ph_\l \circ f_{\l_0} = f_\l \circ \ph_\l$ on $\J_{\l_0}$.
\een 
\edfn

See \cite{MSS83} for the notion of holomorphic motions and of $J$--stability in the setting of analytic families of rational functions.

\section{Models with snowflake tract}\label{s2}

The restriction of an entire function to a logarithmic tract is an example of a model function. The approach here goes the opposite way. We first construct explicit model functions and then, later in Section \ref{s7}, use the  uniform approximation  of  Rempe \cite{Rempe-HypDim2} in order to get entire functions having the same required properties.

According to Definition \ref{tract model}, it suffices to indicate conformal maps $\tau : \Om \to \cH_{\log r}$ or, equivalently, the inverses
$\ph :\cH_{\log r} \to \Om$ in order to define appropriate model functions. Since we will employ the approximation \cite{Rempe-HypDim2}, it is
necessary to define $\ph$ on a larger domain $\hat \cH \supset \cH\supset \cH_{\log r}$ which will be called \emph{extended half plane}. Then, $\Om$ is the tract of the model $f$ and  the larger domain $\hat \Om= \ph (\hat \cH)\supset \Om$ will be called
\emph{extended tract}.

We start by defining the extended half plane $\hat \cH$.
Let  $\sg (t)= -14\log |t| -7$ for $|t|>1$ and extend this function to 
an even and
 $C^\infty$-smooth function $\sg:\R\to (-\infty, 0]$
such that 
$$\sg (0)=0 \quad \text{and} \quad -7\leq \Re (\sg (t))\leq 0 \quad \text{for}\quad |t|\leq 1.$$ 
Consider then
\beq \label{39}
\hat \cH =\{ z=x+iy \; , \;\; x> \sg (y)    \}\supset \cH = \cH_0\,.
\eeq
This domain is a regularized version of the domain used in \cite{Rempe-HypDim2}.
Let $h:\hat \cH \to \cH$ be the conformal map fixing the origin and infinity and such that $h'(\infty)=1$
(see Appendix \ref{appendix} for the existence of $h'(\infty)$). Notice that the symmetry of the graph of $\sg$
implies that $\overline h(z)=h(\ov z)$, $z\in \hat \cH$.

Let $\ph =\tau^{-1} : \hat \cH \to \hat \Om$ be a conformal map fixing the origin and infinity and
 set $\psi =\ph \circ h^{-1}$. We then have the following diagram.
 
 \medskip \smallskip
 
 \begin{center}
 \begin{tikzcd}
  \hat \cH \arrow[r, "\varphi"] \arrow[dd, "h"]&  \hat \Om \\
  &\cup\\
\quad  \qquad\;\;\;  \cH \supset  \cH_{\log r} \arrow[uur, "\psi"]  \arrow[r, "\varphi"] &\Om\\
  \end{tikzcd}
  \end{center}  
\vspace{-0.3cm}
  So,  $\Omega=\ph (\mathcal H_{\log r} )$ does depend on $r$. This number will be taken 
 $r\geq e^2$ and will  be defined in \eqref{x.9}.
   Since we will always have $2\pi i \in \partial \hat \Om$, we can make the additional normalization $\psi (i)=2\pi i$ so that,
all in all,  the conformal map
 $\ph:\hat \cH \to \hat \Om$ is normalized by
 \beq\label{44}
 \ph (0)=0\; , \; \ph(\infty ) =\infty \quad \text{and} \quad \ph (h^{-1}(i)) = 2\pi i\,.
 \eeq

From the particular form of $\hat \cH$ follows that $h$ behaves almost like the identity near infinity. We have collected all required properties of this map in Appendix \ref{appendix}. 

\smallskip


We now define appropriately the domains $\hat \Om$. 
We recall that in this way we also define the domains $\Om = \ph (\cH_{\log r})$, the maps 
$ \tau=\ph^{-1}$ and the model functions $f=e^\tau : \Om \to \D_r^*$ and these are  in fact defined on the larger extended tract $\hat \Om \supset \Om$.

\smallskip

We now describe the particular construction of  the extended tracts $\hat \Om$.
It is based on a modification of a standard snowflake arc $\ga$ attached at the endpoints 
$2\pi i$ and $4\pi i$ and with  the $4^n$ intervals of the $n$-th approximation of length $\lmm_n$ defined as follows. 
Let
\beq\label{1}
\frac 14 < \rho_{min} < \frac{1}{3} < \frac{1}{e} < \rho_{\max } < \frac 12\, ,
\eeq
let $\rho_n\in [\rho_{min}, \rho_{max} ]$, $n\geq 1$, and define then inductively 
$\lmm_n=\rho_n\lmm_{n-1}$, $n\geq 1$, with $\lmm_0=2\pi$. 
If $\rho_n =\frac 13$ for all $n\geq 1$ then $\ga $ is a standard snowflake arc with dimension
$\Theta =\frac{\log 4}{\log 3} $.
The domain $\hat \Om$ is the connected component of the complement  of the curve
\beq\label{quasicircles}
\Ga = \{0\}\cup \bigcup_{n\in \Z} 2^n (-\ga \cup \ga ) 
\eeq
containing the half-line $[10, \infty)$. 
By construction, $0, 2\pi i\in \Ga =\partial \hat \Om$.

Now, with $\tau =\ph^{-1}$ as introduced above and since $\Om=\ph (\cH_{\log r})\subset \ph (\hat \cH)=\hat \Om$,
we get the associated model map 
\beq\label{2}
{ f} =e^{\tau} : \Om \to \D^*_{r} \,.
\eeq
By construction, $ f$ is defined on the larger domain $\hat \Om$. On the other hand, this function is not of disjoint type. 
In order to remedy this it suffices to translate 
the  curve $\Ga$ 
so that, after translation,  $\overline {\hat \Om} \cap \overline \D_{r} =\emptyset$. 
Of course, this implies that $\overline  \Om \cap \overline \D_{r} =\emptyset$ since $\Om \subset \hat \Om$.
Such a disjoint type model is given by
\beq\label{10}
{\bf f}(z) = { f} (z -T)\quad , \quad z-T\in \hat \Om\,.
\eeq
where the precise value of $T$ will be fixed in Section \ref{s7}.

\medskip

In general,
$\Th=\Th_f$ will be the transition parameter of the transfer operator as introduced in \cite{MUpreprint3}
but each time we deal with one of the following examples we will have $\Th =\Th_f=\frac{\log 4}{\log 3}$. 

\bexam \label{exam 1}
Let $\al >1$ and choose the numbers $\rho_n $ such that
\beq \label{lm}
\frac{1}{C_f} \leq n^\al 
 4^n \, \lmm_n^\Theta  \leq C_f\quad , \quad n\geq 1\,,
 \eeq
for some constant $C_f>1$.
\eexam

\brem
Although this will not be used, it is helpful to have in mind that \eqref{lm} allows to show that the Minkowski  dimension of $\ga$ is $\Th =\frac{\log 4}{\log 3}$. 
\erem

\bexam \label{exam 2}
Let $N>1$ and let  $\al > 1$. Set 
$$\rho_n=\frac 1e \quad \text{for}\quad 1\leq n\leq N$$
 and let $\rho_n\in [\rho_{\min } , \frac 13]$ for $n>N$
 such that \eqref{lm} holds for some constant $C_f>1$.
\eexam

\noindent
Clearly, the domains coming from Example \ref{exam 2} are special cases of those in Example \ref{exam 1}.

\

Standard references on quasiconformal mappings are \cite{Ahl06,   AIM-2009, LV-1973}
An important feature is that a $K$--quasiconformal map $\ph : \C\to \C$ is $K'$--quasisymmetric,
with $K'=K'(K)$ depending on $K$ only,
which means that 
$$
|\ph (z_1)-\ph(z_2)|\leq K' |\ph (z_1)-\ph(z_3)| \quad \text{for every} \quad |z_1-z_2|\leq |z_1-z_3|\,.
$$
Moreover, $\ph$ has good H\"older continuity properties:
there are constants $0<c_1\leq  c_2$   such that 
\beq\label{5}
c_1|z_1-z_2|^{K}\leq |\ph (z_1)-\ph (z_2)|\leq c_2 |z_1-z_2|^{1/K} \; \text{ for all} \; z_1,z_2\in \D (0, 2 )
\eeq
and, if $\ph(0)=0$, then
\beq\label{5'}
\D \big(0, c_1 R^{1/K}\big)\subset \ph \big(\D(0,R)\big) \subset \D \big(0, c_2 R^{K}\big)\; \text{ for all} \; R>1 .
\eeq
The first inequality \eqref{5} is mainly Mori's Theorem, for a precise version see Theorem II.4.3 in \cite{LV-1973}.
One can find inequality \eqref{5'}  in \cite[Theorem 3.2]{Martin-1997}.

All these constants do depend quantitatively on each other. In particular, if we deal with a family of uniformly quasiconformal
mappings of the plane, meaning that they are all $K$--quasiconformal for some fixed constant $K$,
and if the maps are normalized, for example by the requirement that \eqref{44} holds,
then the quasisymmetric and the H\"older constants $c_1,c_2$ are also uniform.

\smallskip

A \emph{quasicircle} is the image of a circle or a line by a quasiconformal map of the plane. We only consider unbounded 
quasicircles. Such a curve $\Ga$ is characterized by the important Ahlfors $3$--point condition: there exists $c>0$ such that
$$
|z_1-z_2|\leq c |z_1-z_3| \quad \text{for all}\quad z_1,z_3\in \Ga \; \text{ and }\; z_2\in \Ga (z_1,z_3)
$$
where $\Ga (z_1,z_3)$ is the subarc of $\Ga$ with endpoints $z_1, z_3$. It is a well known fact that all the curves 
defined in \eqref{quasicircles} satisfy this condition uniformly and have uniform quasisymmetric parametrization
(for a proof, see for example Lemma 3.1 in \cite{Rohde-2001}):

\bfact \label{fact quasicircle}
There are constants $c,K'$ depending on $\rho_{max}$ only such that for all choices of $\rho_n\in [\rho_{min}, \rho_{max}]$
the curve $\Ga$ defined in \eqref{quasicircles} satisfies the Ahlfors $3$--point condition with constant $c$ and the natural parametrization of $\Ga$ is a $K'$--quasisymmetry.
\efact

\noindent
Inhere we call \emph{natural parametrization} the map $g:\R\to \Ga$ defined as follows. 
If 
$m\geq 0$ and if $\ga_m$ is the $m$-th approximation of $\ga$ then $\ga_m$ is the union of $4^m$
intervals $I_{m,l}=[a_{m,l}, a_{m,l+1}]$, $l=0,..., 4^m-1$, and we can define $g: [i/2, i]\to \ga$ 
as the continuous extension of the map defined 
by
\beq\label{natural map}
a_{m,l}=g\Big(\frac i2 (1+l/4^m)\Big)
\quad \text{for every} \quad m\geq 0 \text{ and } 0\leq l \leq 4^m\,.
\eeq
The natural parametrization of $\Ga$ will be the unique extension to $i\R$ of this map that satisfies the two relations $g\circ 2=2\circ g$ and $g(-z)=-g(z)$, $z\in i\R$.

The quasicircles  $\Ga$ given by the Examples \eqref{exam 1} and \eqref{exam 2} admit uniformly quasiconformal reflections which allows to show that the corresponding conformal maps 
$\ph:\hat \cH\to \hat \Om$ have normalized and uniformly quasiconformal extension to the plane. 
Also, there are several extensions, such as the one based on the Beurling-Ahlfors extension \cite{BeurlingAhlfors1956}, of a quasisymmetric parametrization of a quasicircle to a quasiconformal map of the plane with control of the constant. Consequently, 
the family of natural parameterizations of the curves $\Ga$
extend to a family of normalized and uniformly quasiconformal mappings of the plane.

\medskip

As a first direct consequence of these properties along with the normalisation \eqref{44}
 we have the following:
 
 \brem \label{normality f} The family of all conformal maps $\ph:\hat{\mathcal H}\to\hat\Omega$ normalized by \eqref{44} (as well as $\psi:\mathcal H\to\hat\Omega$) 
 corresponding to all possible choices of 
 $\rho_n\in [\rho_{min}, \rho_{max} ]$, $n\geq 1$, is a normal family and each limit of a
 convergent sequence of these maps is again a non-constant conformal map.
  \erem
  
  Indeed, the maps $\ph$ are normalized by \eqref{44} and they have uniformly quasiconformal extension to the plane.
 The statement in Remark \ref{normality f} is thus a standard fact for families of normalized uniformly quasiconformal maps
 and, consequently,  Remark \ref{normality f}    not only applies to the conformal maps $\ph, \psi$ but also to their quasiconformal extensions whose convergent subsequences converge uniformly on every compact subset of the plane.
   
 The normality behavior of these maps allows to precise that
  the conformal map $\psi$ reflects the self-similarity of the curve $\Ga$:

\blem\label{4}
If $i\mu= \psi^{-1} (2\, \psi (i)) = \psi^{-1} (4\pi i)$ then  $$\psi \circ \mu = 2 \circ \psi$$ and there exists $1< \mu_{min}\leq \mu_{max} <\infty$
such that for every choice of  $\rho_n\in [\rho_{min}, \rho_{max} ]$, $n\geq 1$, we have
$\mu_{min}\leq \mu\leq \mu_{max}$ .
\elem

\bpf
Since, by construction, $2\hat\Om =\hat \Om$
the map $\psi^{-1} \circ 2\circ \psi $ is a conformal self-map of $\cH$ fixing the origin and infinity
which immediately implies the validity of the functional relation for some real $\mu>1$. 

Set $\mu_{min}=\inf \mu$ where the infimum is taken over all curves $\Ga$ as defined in \eqref{1} and \eqref{quasicircles}.
If $\mu_{min}=1$ then, for every $k\geq 1$, there exists  $\rho_n\in [\rho_{min}, \rho_{max} ]$, $n\geq 1$, such that the associated conformal map $\psi _k$ satisfies the above functional relation with number $\mu_k\in ]1, 1+\frac 1k[$.
We may assume that $\psi_k$ is a converging sequence. Let $\psi$ be the non-constant limit
conformal map of this sequence. Then  $\mu= \frac 1i \psi^{-1} (2\, \psi (i)) >1$ since $ \psi^{-1} ( \psi (i))=i$. 
Finiteness of $\mu_{max}=\sup \mu$ can be shown by a similar normal family argument.
\epf
  
\medskip

A second direct consequence of Remark \ref{normality f} is that the  number $r\geq e^2$ in the definition of the tract $\Om$ can be defined 
such that 
\beq\label{x.9}
\Om=\ph (\cH_{\log r} ) \subset \D_4^* 
\eeq
for all conformal maps $\ph$ of this Section.
We always assume that this is the case.

\medskip

We use several standard notations 
such as  the symbols $A \preceq B$ and $A \asymp B$ which mean that  the ratio $A / B$ is bounded above respectively bounded above and below by constants that do not depend on the particular choice of the numbers
$\rho_n$, some  of them will depend on $\rho_{min} , \rho_{max}$ and parameters like $r$ above. But these are fixed and the same for all models of this paper. In other words, all constants will be uniform
for the family of quasidisks $\Om$, of conformal maps $\ph$ and of models $f$ we consider.

Throughout the text when we refer to \emph{all models of Section \ref{s2}} then this 
refers to the models from Example \ref{exam 1} and Example \ref{exam 2} with fixed numbers $\rho_{min} , \rho_{max}$
according to \eqref{1}.

\section{Estimating the transfer operator}\label{s3}

 The paper  \cite{MUpreprint3} contains a complete treatment of the thermodynamic formalism of disjoint type models and functions.
 We now collect some properties of the central tool of this theory, the transfer operator. In order to do so, we consider a model $f=e^{\tau}:\Om \to \D_r^*$. Typically, $f$ is one of the examples of the previous section but it can also be the restriction of a convenient  entire function to its logarithmic tract.

In the sequel we will work with a particular Riemannian, in fact the cylindrical, metric 
$|dz|/|z|$.
The derivative of a holomorphic function $h$ calculated with respect to this metric  at a point $z$ such that $h(z)\neq 0$  is denoted by $|h'(z)|_1$ and is given by the formula
\beq \label{6'}
|h'(z)|_1=|h'(z)|\frac{|z|}{|h(z)|}.
\eeq
 Given a real number $t\geq 0$, we define the transfer operator $\pft$ by the usual formula:
\beq\label{6}
\pft g (w):= \sum_{f(z)=w} |f'(z)|^{-t}_1 g(z) \quad \text{for every}\quad w\in \D_{r}^*
\eeq
where $g$ is any function in $\cC_b (\ov\Om)$, the vector space of all continuous bounded functions defined on $\Om$. The norm on this space, making it a Banach space, will be the usual sup-norm $\|\cdot\|_\infty$.
 Note that if $w\in\D_{r}^*$, then $f^{-1}(w)\sbt\Om$ and, by the disjoint type assumption, $\Om \subset \D_{r}^*$.
 Thus $|f'(z)|_1$ is well defined for all $z\in f^{-1}(w)$ and, in consequence, $\pft g (w)$ is well defined for all
 $w\in \D_{r}^*$ provided the series is convergent. Since we work with quasidisk tracts the whole scope of  \cite{MUpreprint3} applies and we know in particular that there is a number $\Th=\Th_f \in [1, 2]$, called transition parameter, such that  the series defining  $\pft$ is convergent if $t> \Th$ and diverges if $t<\Th$.  
 \bdfn \label{7}
 The function $f$ is of convergence type if the series defining $\pft$ converges for $t=\Th$.
 \edfn

The reader should have in mind the following fact (which is a very particular case of Theorem 4.1 in \cite{MUpreprint3}):

\bthm \label{MU Thm 4.1}
Let $f$ be a model or an entire function of class $\cB$ having one (or finitely many) tracts all of them being
 quasidisks.
Assume that $f$ is of disjoint type. Let $t>0$ and suppose that there exists $w_0\in \D_{r}^*$ 
 such that
$\pft \1 (w_0) < \infty$. Then the series defining $\pft \1$ is uniformly convergent meaning that 
$\pft$ is a bounded operator of  the space $\cC_b (\ov\Om)$.
\ethm

\subsection{Integral means}

The transition parameter $\Theta$ is precisely determined by the geometry of the boundary of the tract $\Om$ near infinity.
For this one considers rescalings of the conformal map $\ph$ given by
\beq\label{8}
\ph_T:=\frac{1}{|\ph(T)|} \, \ph \circ T \; \; , \;\; T\geq 1 \; .
\eeq
The map $\ph_T$ is defined on $ T^{-1}\hat  \cH$. In particular, all the maps $\ph_T$, $T\geq 1$, are defined on the half space $\cH$.

Let us consider again one of the examples introduced in Section \ref{s2}.
By self-similarity of the tracts and in view of Lemma \ref{4}, it suffices to consider only values $T=\mu^n$, $n\geq0$.
Considering now integral means of the rescalings of $\ph$ we get the required information about the  geometry of the boundary of the tract $\Om$ near infinity: 
\beq\label{9}
\b _\infty (t) =\limsup_{n\to\infty} \frac{\log \int _{\mu^{-1}\leq |y| \leq 1} |\ph _{\mu ^n}' (\mu^{-n} +iy)|^t dy}{n\log \mu}
\quad , \quad t\geq 0 \,.
\eeq

Of particular importance is the function $t\mapsto b_\infty(t)=\b_\infty (t)-t+1$. Following \cite{MUpreprint3},
this function always has a smallest zero $ \Th_f> 0$ and, in the good cases, it has a unique zero and is negative in $(\Th_f , \infty )$. In this latter case,  the function $f$ is said to have \emph{negative spectrum} and then $\Th_f$ is the transition parameter of the transfer operator
(see Proposition 5.6 and Theorem 4.4 in  \cite{MUpreprint3}).

In our case,  $\Om$ is a  quasidisk and functions with quasidisk tracts have negative spectrum (see Section 5 of \cite{MUpreprint3}).
This can be compared to the classical case of a conformal map of the unit disk onto a bounded quasidisk and with $\b_\infty$ the standard integral means function. 
There, Pommerenke has shown that $b_\infty$ has a unique zero which 
 is the Minkowski dimension of the boundary of the quasidisk (see Corollary 10.18 in \cite{PommerenkeBook}). 
 We will show in Theorem \ref{32} that the transition parameter of the transfer operator for our examples is $\Theta=\frac{\log 4}{\log 3}$, hence the Minkowski dimension
of the snowflake curves in Section \ref{s2}.
This means that the models or entire functions we deal with in the present paper have
negative spectrum and their disjoint type versions are in the class $\cD$ defined in \cite{MUpreprint3}. So the whole scope of that paper
 applies.

\subsection{The transfer operator of the models}\label{section 4.3}
We now come back to the models introduced in Section \ref{s2} and give precise estimates for the transfer operator of these models. The first step, which expresses $\pft\1(w)$ as an integral, follows Section 4 of \cite{MUpreprint3} and so we can allow us to present only the essential steps.

Let $f$ be one of the models introduced in Section \ref{s2}. 
Let $w\in \D^*_{r}$ and set $x=\log |w| > \log r\geq 2$. Fix also 
\beq\label{11}
\text{$j\geq 1\;\; $ maximal such that $\;\;\mu^{j-1}\leq x$.}
\eeq
For $z\in f^{-1}(w)\in \Om$ we have
$
|f'(z)|_1 = \left| \frac{\ph (\xi)}{\ph'(\xi)}\right|
$ where $\xi=\ph^{-1} (z)$. Hence,
$$
\pft \1 (w) = \sum_{\xi\in\exp^{-1}(w)} \left| \frac{\ph' (\xi)}{\ph(\xi)}\right|^t = \sum_{\xi\in\exp^{-1}(w)} \left| (\log \ph)'(\xi)\right|^t 
$$
and thus, using bounded distortion,
\begin{align*}
\pft \1 (w)& \asymp \int _\R  \left| (\log \ph)'(x+iy)\right|^t dy\\
&=  \int _{-\mu^j\leq y \leq \mu^j}  \left| \frac{\ph'}{\ph} (x+iy)\right|^t dy +\sum_{n\geq 1}
\int _{I_{n+j}}  \left| \frac{\ph'}{\ph}(x+iy)\right|^t dy
\end{align*}
where  $I_m= [-\mu^m , -\mu^{m-1}]\cup[\mu^{m-1}, \mu^m]$. The first term can be estimated as follows.
By quasisymmetry of $\ph$, since $\ph(0)=0$ and since $|y|\leq \mu^j \asymp x$ we have 
$$|\ph (x+iy)|\asymp diam \Big(\ph\big(\D (x+iy , \frac x2 )\big)\Big)\,.$$
On the other hand, $diam (\ph(\D (x+iy , \frac x2) )) \asymp |\ph ' (x+iy)|x$ because of bounded distortion. Consequently,
$$
\int _{-\mu^j\leq y \leq \mu^j}  \left| \frac{\ph'}{\ph} (x+iy)\right|^t dy  \asymp x^{1-t}\,.
$$
The integrals over $I_{n+j}$ can be estimated using the rescalings $\ph _T$ introduced in \eqref{8}
with $T={\mu^{n+j}}$ where, we recall, $j$ comes from \eqref{11}.
Again quasisymmetry of $\ph$ and the fact that $\varphi(0)=0$  show that $|\ph (x+iy)|\asymp |\ph( \mu ^{n+j})|$, $y\in I_{n+j}$.
It thus follows from a simple change of variable combined with bounded distortion that
\begin{align*}
\int _{I_{n+j}}  \left| \frac{\ph'}{\ph}(x+iy)\right|^t dy=&
 \int _{I_{0}}
\frac{|\ph'(\mu^{n+j}(   \mu^{-n-j}x+i u))|^t}{|\ph( \mu ^{n+j})|^t} \mu^{n+j}du \\
\asymp &
\int _{I_{0}}
\frac{|\ph'(\mu^{n+j}(   \mu^{-n}+i u))|^t}{|\ph( \mu ^{n+j})|^t} \mu^{n+j}du .
\end{align*}
Since $\ph_T' = \frac{T}{|\ph (T)|}\, \ph' \circ T$ we get, taking $T= \mu^{n+j}$,
$$
\int _{I_{n+j}}  \left| \frac{\ph'}{\ph}(x+iy)\right|^t dy\asymp
(\mu^{n+j})^{1-t} \int _{I_0}|\ph'_{\mu^{n+j}}(\mu^{-n}+iy)|^tdy.
$$
Let $r_n = r_n(x) = (n+j+\mu^j)\mu^{-(n+j)}\asymp \mu^{-(n+j)}dist(x+iy, \partial \hat \cH )$, $y\in I_{j+n}$. 
An elementary calculation shows that
 \beq \label{19} 
\mu^{-n}\leq r_n \preceq n \mu^{-n} \quad , \quad n\geq 1\,.
\eeq

\medskip

\noindent
Choose a maximal number of points $y_{n,k}$ in
$\mu^{-n}+ i  I_0$ such that $$sign (\Im (y_{n,k} ) )= sign (k)$$ and such that two consecutive points have distance $r_n$. 
Then,
\begin{align*}
\int _{I_{n+j}}  \left| \frac{\ph'}{\ph}(x+iy)\right|^t dy&\asymp
(\mu^{n+j})^{1-t} \sum_k|\ph'_{\mu^{n+j}}(y_{n,k})|^t r_n\\
&\asymp (n+j+x)^{1-t} \sum_k \left( |\ph'_{\mu^{n+j}}(y_{n,k}) |r_n\right)^t \,.
\end{align*}
Finally, it is convenient to replace $\ph$ by $\psi = \ph \circ h^{-1}$. We have
$$
\ph_{\mu^m}=\frac{1}{|\ph (\mu^m )|}\ph \circ \mu^m=\frac{2^m}{|\ph (\mu^m )|} \left( \frac{1}{2^m}\psi \circ \mu^m\right)
\left(\mu^{-m}\circ h \circ \mu^m \right).
$$ 
The first factor is approximately equal to $1$ since $\ph$ is a quasisymmetry with $\ph(0)=0 $, since
$\mu^m\asymp |h^{-1} (i\mu^m)|$ and since 
$\ph (h^{-1} (i\mu^m))= \psi (i\mu^m)=2^m \psi(i)=2^m 2\pi i$ by Lemma \ref{4}. The same lemma implies that
the second term equals $\psi $ and in order to estimate the last factor we introduce $$h_m=\mu^{-m}\circ h \circ \mu^m.$$ It follows from Proposition \ref{41} in Appendix \ref{appendix}  that $|h_m'|\asymp 1$. Therefore,
\beq\label{12}
|\ph' _{\mu^m}|\asymp  |\psi ' \circ h_m| 
\eeq
and, injecting this in the above expression, we get
$$
\int _{I_{n+j}}  \left| \frac{\ph'}{\ph}(x+iy)\right|^t dy\asymp 
(n+j+x)^{1-t} \sum_k \left( |\psi' (h_{n+j}(y_{n,k})) |r_n\right)^t 
$$
Set $z_{n,k}= h_{n+j}(y_{n,k})$. We will see in Lemma \ref{15} below that $dist (z_{n,k}, \partial\mathcal H)\asymp r_n$.  We get all in all
\beq \label{14}
\pft \1 (w) \asymp x^{1-t} \left[1+ \sum_{n\geq 1}  \left\{\Big(1+\frac{n+j}{x}\Big)^{1-t} \sum_k \left( |\psi' (z_{n,k}) |r_n\right)^t  \right\} \right]
\eeq
for all $w\in \D^*_{r}$.
The factor $ |\psi' (z_{n,k}) |r_n$ has obvious geometric meaning.  Indeed, assume that $Q_{n,k}\subset \cH$ is a rectangle  containing $z_{n,k}$, and such that 
\beq\label{26'}diam(Q_{n,k})\asymp r_n \quad \text{and } \quad dist(Q_{n,k},\partial \cH) \asymp r_n.\eeq
 Set 
\beq\label{26}
\cW_{n,k} = \psi (Q_{n,k}).
\eeq
 Then the following statement immediately follows from bounded distortion and \eqref{14}.

\bprop\label{13}
With the previous notations we have
$$
\pft \1 (w) \asymp ( \log |w| )^{1-t} \left[1+ \sum_{n\geq 1}  \left\{\Big(1+\frac{n+j}{x}\Big)^{1-t} \sum_k (diam \,\cW_{n,k})^t  \right\} \right]
$$
for all $w\in \D^*_r$ and with comparability constants uniform for all models of Section \ref{s2} but depending on the multiplicave constants in \eqref{26'}.
\eprop

In order to  exploit this we have to define properly the rectangles $Q_{n,k}$. We first need a technical result.

\blem\label{15}
There exists $\ka \geq 1$, independent of $n$ and $k$  such that the following properties hold:
\ben
\item \label{15.1} For  all $(n,k)$,
$$\frac{r_n}{\ka}\leq \Re (z_{n,k}) \leq \ka r_n \; .$$
\item \label{15.2} $\ka r_{n+1} > \frac{r_n}{\ka}$ for every $n\geq 1$ and $x > \log r \geq 2$.
\item \label{15.3} $|z_{n, k+1}-z_{n,k}|\leq\ka r_n$ for every $n\geq 1$ and $k>0$ and the analogue statement also holds if $k<0$.
\item \label{15.4}  Let $L= L_1$ be a common bilipschitz constant for the maps $h_m$, $m\geq 1$ (see Proposition \ref{41} in  Appendix \ref{appendix}).  Then
$$ \frac{1}{L\mu_{max}}  \leq|\Im z_{n,k}| \leq L \;\;\text{ for all $(n,k)$.}$$
\een
\elem

Given \eqref{15.4}, it is appropriate to define the values  
$$\text{$s_{imag}= \frac{1}{L\mu_{max}}$ and $S_{imag}=L$ .}$$

\bpf
For every $m\geq 1$, the map $h_m: \mu^{-m}\hat \cH\to  \mu^{-m}\cH =\cH$ is conformal and thus a hyperbolic isometry.
This implies that, for every $\xi \in  \mu^{-m}\hat \cH$,
$$
dist(\xi , \partial  \mu^{-m}\hat \cH) |h_m'(\xi )|\asymp dist ( \h_m(\xi ),\partial \cH) = \Re( \h_m(\xi )).
$$
Taking $\xi=y_{n,k}$ and $m=n+j$ we get
$$
\Re (z_{n,k}) \asymp dist(y_{n,k}, \partial  \mu^{-n-j}\hat \cH) |h_{n+j}'( y_{n,k})| \asymp r_n
$$
since, by Proposition \ref{41}, $ |h_{n+j}'( y_{n,k})| \asymp 1$. This shows Item \eqref{15.1}. 

Item \eqref{15.2} follows from the estimate
$$
\frac{r_n}{r_{n+1}}=\frac{(n+j+\mu^j)\mu^{-(n+j)}}{(n+1+j+\mu^j)\mu^{-(n+1+j)}}
= \mu \left( 1- \frac{1}{n+1+j+\mu^j}\right)\leq \mu\leq \mu_{max}
$$
with $\mu_{max}$ from Lemma \ref{4}.

Since the maps $h_m$ are bilipschitz uniformly with $m$, we have  the following: if $k_1,k_2$ have same sign then
 \beq\label{20} | z_{n,k_1}-z_{n,k_2}| \asymp | y_{n,k_1}-y_{n,k_2}| = |k_1-k_2|r_n\eeq
In particular, $|z_{n, k+1}-z_{n, k}| \asymp |y_{n, k+1}-y_{n, k}|=r_n$ which shows Item  \eqref{15.3}.

Finally, Item \eqref{15.4} follows from Lemma \ref{BL appendix}.
\epf

\medskip

\noindent
Let $\ka\geq  1$ be given by Lemma \ref{15}. This number being fixed, we can now define the rectangles $Q_{n,k}$
around $z_{n,k}$ as follows:
$$
Q_{n,k}=\Big\{\frac{r_n}{\ka}\leq \Re (\xi ) \leq \ka r_n   \text{ , }  |\Im (\xi-z_{n,k})|\leq \ka r_n\Big\}\cap \Big\{ s_{imag}\leq |\Im (\xi )| \leq  S_{imag}\Big\}.$$
Notice that \eqref{26'} is satisfied since $z_{n,k}\in Q_{n,k}$ by Item \eqref{15.1} and Item \eqref{15.4} of Lemma \ref{15}.

\medskip

\blem\label{16} Let $\ka\geq  1$ be given by Lemma \ref{15} and let $Q_{n,k}$ be defined as above. Then:
\ben
\item $\bigcup_{n,k}Q_{n,k} \subset \cU_{ext} =\Big\{0<\Re (\xi ) < S_{real} \; , \; \frac{s_{imag}}{2}< |\Im (\xi )| < 2 S_{imag}\Big\}$ where 
\begin{equation*}
  S_{real} =  2\ka \Big( 1+\frac{1}{\log \mu_{min}}\Big) .
\end{equation*}
\item There exist $\d_{real}>0$ and $\mu_{min}^{-1} < \d_{imag}^- <\d_{imag}^+ <1$ such that
$$\quad \bigcup_{n,k}Q_{n,k} \supset \cU_{int}= \Big\{0<\Re (\xi ) < \d_{real} \; , \;  \d_{imag}^-< |\Im (\xi )| <\d_{imag}^+\Big\}.$$ 
\een
\ben
\setcounter{enumi}{2}
\item The collection $\{ Q_{n,k}\}$ has bounded overlap: there exist $B\geq 1$ such that for every $(n_0,k_0)$ there exist at most $B$ indices $(n,k)$ such that $$Q_{n,k}\cap Q_{n_0,k_0}\neq \emptyset.$$
\een
 Again, all the involved constants are uniform. In particular, the sets $\cU_{int}$ and $ \cU_{ext}$ do not depend on the model $f$.
 \elem

\bpf An elementary calculation shows that $\sup_{a\geq 1} \frac{a}{\mu^a}\leq \frac{1}{\log \mu}\leq \frac{1}{\log \mu_{min}}$.
Combined with the definition of $r_n$ and of  $S_{real}$ we get for every $\xi \in \bigcup_{n,k}Q_{n,k} $ that $\Re {\xi }< S_{real}$.
Item (1) follows since the assertion concerning the imaginary part is obvious given the definition of the sets $Q_{n,k} $.

The second item can be shown as follows. Fix arbitrarily $\d_{imag}^- , \d_{imag}^+ $ such that $\mu_{min}^{-1} < \d_{imag}^- <\d_{imag}^+ <1$. By the definition of the points $y_{n,k}$ there exists $k_1,k_2>0$ such that 
$\Im (y_{n,k_1} ) < \mu^{-1}+r_n$ and $\Im (y_{n,k_2} ) > 1-r_n$. The bilipschitz property in Lemma \ref{BL appendix} implies thus that
$$
\Im (z_{n,k_1} ) \leq L_n (\mu^{-1}+r_n)\leq L_n(\mu_{min}^{-1}+r_n) \; \text{ and }\; 
\Im (z_{n,k_2} ) \geq \frac 1{L_n}(1-r_n). 
$$
Notice that $(L_m)_m$ is a decreasing sequence with limit $1$. On the other hand, $r_n\to 0$ and thus there exists $n_{min}$, which does not depend on the model $f$,  such that  
$$
L_{n}(\mu_{min}^{-1}+r_{n}) < \d_{imag}^-    \; \text{ and }\;  \frac 1{L_n} (1-r_n)>  \d_{imag}^+
\; \text{ for every }\; n\geq n_{min}.
$$
If we combine this with the definition of the sets $Q_{n,k}$ and Item \eqref{15.3} of Lemma \ref{15} then this gives
$$
V_n=\Big\{\frac{r_n}{\ka }< \Re (\xi ) < \ka r_n  \; , \;  \d_{imag}^-   < \Im (\xi )<  \d_{imag}^+ \Big\}
\subset  \bigcup_{n,k}Q_{n,k}
$$
for all $n\geq n_{min}$. 
Given \eqref{15.2} of Lemma \ref{15}, the set 
$$
\bigcup_{n\geq n_{min}}V_n
 $$
 covers $ \cU_{int}$ if we set $\d _{real} = \ka _{r_{n_{min}}}$.

\medskip

We are left to show that the collection $\{ Q_{n,k}\}$ has bounded overlap.
To start with, suppose that $ n < m$ and  $(n,k)$, $(m,l)$ are such that  $Q_{n,k}\cap Q_{m,l}\neq \emptyset$. Then 
necessarily $ \ka r_{m}\geq \frac{r_n}{\ka}$. But
$$
\frac{r_n}{r_m}= \frac{(n+j+\mu^j)\mu^{-n-j}}{(m+j+\mu^j)\mu^{-m-j}}
=\mu^{m-n}  \frac{n+(j+\mu^j)}{m+(j+\mu^j)}\geq \mu^{m-n}  \frac{n}{m}.
$$
Put $\De =m-n\geq 1$. Clearly $\frac n m =\frac n{n+\De}=\frac{1}{1+\De / n}\geq \frac 1{1+\De}\geq \frac1{2 \De}$
so that we get altogether the condition
$$\ka ^2 \geq \frac 12 \frac{ \mu ^\De}{\De}\geq \frac 12 \frac{ \mu_{min} ^\De}{\De} $$
which shows that there is a constant $B_1=B_1(\ka )$ such that $\De = m-n\leq B_1$.

Now, let $\xi\in \cH$, fix $n$ and 
consider  $k$ such that $\xi \in Q_{n,k}$. Then
$$|\xi - z_{n,k}|\leq diam ( Q_{n,k}) \leq \Big(\ka-\frac 1 \ka\Big)r_n+2\ka r_n = \Big(3\ka-\frac 1 \ka\Big)r_n.$$
 It follows from \eqref{20} that this can happen for at most 
$$B_2= 2\Big(3\ka-\frac 1 \ka \Big) L$$
indices $k$ where $L$ is the bilipschitz constant involved in  \eqref{20}. 

In conclusion, $\xi\in Q_{n,k}$ can happen for
at most $B_1$ different indices $n$ and, for every fixed $n\geq 1$, there are at most $B_2$ indices $k$ such that 
 $Q_{n,k}$ contains $\xi$. Therefore, the collection $\{ Q_{n,k}\}$ has bounded overlap with constant $B=B_1B_2$.
\epf

\section{Whitney decompositions}\label{s4}

In order to estimate the transfer operator via the sets $(\cW_{n,k} )$
 we will compare them to Whitney 
decompositions that reflect the geometry of the snowflake curve. 

Whitney coverings are standard. Here we use a slight modification of the usual notion.
The following definition applies to more general open sets $\cV$ but in this paper we will
take
$\cV = \psi (\cU)$
where $\cU$ is one of the sets  $\cU_{int}, \cU_{ext}$ of Lemma \ref{16}
and where $\Upsilon = \psi (\partial \cU \cap i\R)\subset \partial \cV\cap \partial \hat \Om= \partial \cV\cap \Ga$.

\bdfn \label{29}
A collection $(\wh_{m,l})$
 of sets is a  Whitney covering of $\cV$  with respect to $\Upsilon \subset \partial \cV$ if
 the following holds:
\ben
\item $\cV\subset \bigcup \wh_{m,l} $ and $\cV\cap  \wh_{m,l}\neq \emptyset$ for all $(m,l)$.
\item The sets $\wh_{m,l}$ have bounded overlap: there exists $\cB\geq 1$ such that
for every $(m_0,l_0)$ there exist at most $\cB$ indices $(m,l)$ such that $$ \wh_{m,l}  \cap \wh_{m_0,l_0} \neq\emptyset\,.$$
\item The sets $\wh_{m,l}$ are closures of Jordan domains, they are
uniformly round and of diameter comparable to the distance to the boundary. The later two conditions mean that 
 there exists $a>0$ and disks $\D(z_{m,l}, r_{m,l})$ such that the following holds for every $(m,l)$:
\beq\label{29Nov 2}
\D(z_{m,l}, r_{m,l})\subset \wh_{m,l}\subset \D(z_{m,l}, r_{m,l}/a)
\eeq
 and
\beq\label{29Nov 3}
a \, diam (\wh_{m,l}) \leq dist(\wh_{m,l}, \Upsilon) \leq \frac 1a diam (\wh_{m,l}) \,.
\eeq
\een
\edfn

\bfact\label{24}
The Whitney covering property is a conformal, even quasiconformal, invariant.
Indeed, quasiconformal mappings preserve the roundness condition \ref{29Nov 2}
(with new constant $a'$ depending on $a$ and on the quasiconformal constant only)
and \eqref{29Nov 3} is also preserved thanks to an estimate of Gehring and Osgood  \cite{GO79} for the quasihyperbolic distance
(see the explanation by Koskela in \cite[p.210]{QC and Analysis}).
\efact

\subsection{Geometric Whitney covering}
Let again $\hat \Om$ be one of the domains of Section \ref{s2}.
Consider now $g$ a quasiconformal map of the plane such that
$g(\cH )=\hat\Om$ and such that $g$ reflects the geometry of the snowflake curve $\Ga$.
It is a quasiconformal extension of the natural parametrization of $\Ga $ as explained in Section \ref{s2}
and it satisfies the relation \eqref{natural map}. We use this map to produce coverings of the sets
$$
\cV_{int} =\psi(\cU_{int}) \quad \text{and of} \quad \cV_{ext}=\psi(\cU_{ext}) .
$$
In the following, $\cV$ is one of the sets $\cV_{int}, \cV_{ext}$ and we
 recall from Lemma \ref{16} that $\cU_{int}, \cU_{ext}$ do not depend on the model, hence on $\Ga$.

Consider a standard decomposition of $\cH$ given by
$$
{\bf Q}_{m,l} =\Big\{4^{-(m+1)}\leq \Re \xi \leq 4^{-m}  , \;\; l\, 4^{-m} \leq \Im \xi \leq (l+1) 4^{-m} \Big\}
$$
and set
$$
\wh_{m,l} = g({\bf Q}_{m,l} ) \;\; , \;\;\; m ,  l\in \Z\,.
$$
By Fact \ref{24}, the collection of all $(\wh_{m,l} )$ such that $\wh_{m,l} \cap \cV \neq \emptyset$ is a Whitney covering of $\cV$ with respect to $\Upsilon= \psi (\partial \cU \cap i\R)$.
This covering reflects the geometry of the snowflake, as explained in Lemma~\ref{35} below. As always, the constants
in this result do not depend on the particular snowflake chosen out of the family described in Section \ref{s2}.

\blem\label{25} For every set $\wh_{m,l} $ of this Whitney covering of $\cV$ with respect to $\Upsilon$ we have
$diam \wh_{m,l} \asymp  \lm $, there exists $K\geq 1$ such that
\beq\label{25.1}
4^{-mK}\preceq diam \wh_{m,l} \preceq 4^{-m/K}
\eeq
and, for some $m_0\geq 1$, the number of sets $\wh_{m,l}$ of level $m\geq m_0$ is
\beq\label{25.2}
\# \{l\; , \;\wh_{m,l}\cap \cV \neq \emptyset\} \asymp 4^m
\eeq
where the involved equivalence constants do only depend on the set $\cV=\cV_{int}$ or $\cV=\cV_{ext}$ respectively.
\elem

\bpf
The relation  $diam \wh_{m,l} \asymp \lm $ follows from the fact that the quasiconformal map $g$ is quasisymmetric
and \eqref{25.1} is a consequence of the H\"older continuity  \eqref{5}. The statement concerning the number of sets of a given level $m$ is clear and the involved constants are independent of the model since the sets $\cU_{int}, \cU_{ext}$ do not depend on them.
\epf

\subsection{Conformal Whitney covering} The  covering $(\cW_{n,k}=\psi (Q_{n,k}))$ has been introduced in \eqref{26}.

\blem\label{27}
The sets $(\cW_{n,k})$, $(n,k)$ such that $Q_{n,k}\cap \cU_{int} \neq \emptyset$,
are a Whitney covering of $\cV_{int}=\psi(\cU_{int})$ with respect to $\Upsilon_{int}= \psi (\partial \cU_{int} \cap i\R)$. In addition, 
there exists $K\geq 1$ such that
\beq\label{23}
 r_n^K \preceq diam \cW_{n,k} \preceq r_n^{1/K}
\eeq
\elem

\bpf
By Fact \ref{24}, it suffices to verify that $(Q_{n,k})$
is a Whitney covering with respect to $ \partial \cU_{int}\cap i\R$. But this we already checked in Section \ref{s3} (see \eqref{26'} and Lemma \ref{16}).

It remains to justify the inequalities in \eqref{23}. But they follow from
$diam \; Q_{n,k} \asymp r_n$ and, again, from the H\"older property  \eqref{5}.
\epf

\subsection{Comparing the coverings} In view of estimating the series in Proposition \ref{13} we now compare the
geometric and conformal Whitney coverings.

\blem\label{28} There exists a constant $B_*$
such that for every $(n,k)$ (or $(m,l)$)  there are at most $B_*$  indices $(m,l)$ (respectively $(n,k)$) such that
\beq\label{30} \cW_{n,k} \cap \wh_{m,l} \neq \emptyset \,.\eeq
\elem

\bpf 
First of all, there exists $a>0$
 such that every set $\cW_{n,k}$ and $\wh_{m,l}$ contains respectively a ball $B_{n,k}$, 
${\bf B}_{m,l}$ of radius $a\, diam\cW_{n,k}$, $a\, diam \wh_{m,l}$ 
Again, this constant $a$ is independent of the model of Section \ref{s2} since, by uniform quasiconformality, the sets $ \cW_{n,k} , \wh_{m,l}$ are uniformly round. We recall that this means that the roundness condition \eqref{29Nov 2} is satisfied
for some fixed constant $a>0$.

Both coverings being Whitney, \eqref{30} implies $diam \cW_{n,k} \asymp diam \wh_{m,l}$. Therefore, there exists $A>1$
such that, whenever \eqref{30} holds,
$$
B_{n,k}\subset \cW_{n,k} \subset  \D(w_{m,l}, A \, diam \wh_{m,l})
$$
where $w_{m,l} \in \wh_{m,l}$ is any arbitrary point. 
The conclusion comes now from the bounded overlap property combined with a volume comparison argument. Clearly in this argument we can exchange the role of the two coverings and thus the proof is complete.
\epf

We also have to compare the levels $n$ and $m$ for sets $\cW_{n,k}$ and $\wh_{m,l}$ that intersect. This is not possible for general domains but here
we deal with quasidisks and have  good H\"older estimates.

\blem\label{31} There exists a constant $b>0$, still independent of the model,
such that for every $(n,k)$ and $(m,l)$ for which \eqref{30} holds we have
$$
b\, n \leq m \leq \frac 1b n \; .
$$
\elem

\bpf
Assume $(n,k)$ and $(m,l)$ are such that \eqref{30} holds. Then $diam\, \cW_{n,k}\asymp diam\, \wh_{m,l}$.
It follows from Lemma \ref{25} and from \eqref{23} that 
$$
4^{-mK} \preceq r_n^\frac{1}{K} \quad \text{and}\quad r_n ^K \preceq 4^{-\frac mK}\,.
$$
Concerning $r_n$, we use now the estimate \eqref{19}. Combined with the previous one it gives
$$
4^{-mK} \preceq n^\frac{1}{K} \mu^{-\frac nK}\leq n^\frac{1}{K} \mu_{min}^{-\frac nK}\quad \text{and}\quad \mu_{max}^{-nK} \leq  \mu^{-nK} \preceq 4^{-\frac mK}
$$
from which the assertion easily follows.
\epf

\section{Models of convergence type}\label{s5}

Let $f$ be a model of Section \ref{s2} with tract $\Om=\varphi(\mathcal H_{\log r})$ given by Example \ref{exam 1}. We recall that in this case
\beq\label{35}
\lm^\Th \asymp 4^{-m} \frac{1}{m^\al} \quad , \quad m\geq 1\, \quad \alpha>1 \,,
\eeq
where the involved multiplicative constant $C_f$ does depend on the model $f$.

\bthm\label{32}
The transfer operator $\pft$ of $f$ is of convergence type (with $\Th =\log 4 /\log 3$) and there exists $M=M_f\geq 1$
such that
$$
\pft \1 (w) \leq M^t ( \log |w| )^{1-t} 
$$
for every $w\in \D^*_{r}$ and every $t\geq \Th$.
\ethm

\brem\label{32'}As explained in Section 8 of \cite{MUpreprint3}, Theorem \ref{32} implies that for these models the full thermodynamic formalism holds for all $t\geq \Th$ so also in the
particular case when $t=\Th$ equals the transition parameter. 
\erem

\bpf[Proof of Theorem \ref{32}]
From Proposition \ref{13} we have a precise estimate of $\pft$ which implies
$$
\pft \1 (w) \preceq ( \log |w| )^{1-t} \left[1+ \sum_{n\geq 1}   \sum_k (diam \,\cW_{n,k})^t  \right]\,
$$
since $t>1$.
Take $\cU=\cU_{ext}$ and remember from
Lemma \ref{16} that $\cU$ contains all the sets $Q_{n,k}$, hence $\bigcup_{n,k}\cW_{n,k} \subset \cV_{ext}=\psi (\cU_{ext})$.
Set $\cI =\{(m,l)\, ; \; \wh_{m,l}\cap \cV_{ext}\neq \emptyset\}$ so that $\{\wh_{m,l}\; ,  (m,l)\in \cI\}$ is a Whitney covering of $\cV_{ext}$
with respect to $\Upsilon=\psi (\partial \cU \cap i\R)$.
In particular, for every $(n,k)$ there exists $(m,l)\in \cI$ such that 
$$
\cW_{n,k}\cap \wh_{m,l} \neq \emptyset \; \text{ and }\; diam\, \cW_{n,k} \leq C diam\, \wh_{m,l}
$$
for some uniform constant $C$.
 It thus follows from Lemma \ref{28} that 
\beq\label{34}
\sum_{n\geq 1}   \sum_k (diam \,\cW_{n,k})^t \leq  B_* C^{t}\sum_{(m,l)\in \cI}    (diam \,\wh_{m,l})^t .
\eeq
We have $diam \,\wh_{m,l} \asymp \lm$ (Lemma \ref{25}) which, along with \eqref{25.2} of Lemma \ref{25} and \eqref{35}, implies that for every $t\geq \Th$
$$
\pft \1 (w) \preceq ( \log |w| )^{1-t} \sum_{m\geq 1} 4^m \left( C_f \, 4^{-m}\frac{1}{m^\al }\right)^{t/\Theta}
\leq C_f^{t/\Theta}M'( \log |w| )^{1-t}
$$
where $M'= \sum_{m\geq 1} \frac{1}{m^\al } <\infty$.

It remains to show that $\pft\1 (w)=\infty$ for $t<\Theta $ and for some $|w|> r\geq e^2$.  
We first provide an appropriate lower bound for the transfer operator starting again from Proposition \ref{13}.
The expression there gives, for every $w\in \D_r^*$ and still with $x=\log |w|$,
\begin{align*}
\pft \1 (w)& \succeq ( \log |w| )^{1-t} \sum_{n\geq 1}  \left\{\left(1+\frac{n+j}{x}\right)^{1-t} \sum_k (diam \,\cW_{n,k})^t  \right\}.\end{align*}
Since
$$
1+\frac{n+j}{x}\leq 1+n+\frac{j}{\mu^{j-1}}\leq 1+n+\frac{\mu}{\log \mu}
$$
we have
$$
\pft \1 (w) \succeq x^{1-t} \sum_{n\geq 1}  \left\{\left(1+\frac{\mu}{\log \mu}+n\right)^{1-t} \sum_k (diam \,\cW_{n,k})^t  \right\} .
$$
Let $0<t<\Th$ and let
$\ep >0$ such that $t'=t+\ep <\Th$.  By \eqref{23} and \eqref{19}
$$diam\, \cW_{n,k} \preceq r_n^{1/K} \preceq n^{1/K} \mu^{-n/K}$$
and thus
$$
c_t=\inf _{n\geq 1} \, (1+\frac{\mu}{\log \mu}+n)^{1-t} (diam \,\cW_{n,k})^{-\ep} \succeq \min _{n\geq 1} \mu^{n\frac{\ep}{K}}\frac{(1+\frac{\mu}{\log \mu}+n)^{1-t}}{n^{\ep /K}} >0.
$$
Injecting this in the lower estimate of $\pft \1$ gives
$$
\pft \1 (w)\succeq  c_t\,  x^{1-t} \sum_{n\geq 1}  \sum_k (diam \,\cW_{n,k})^{t '}.
$$
In view of Lemma \ref{16}, the sets $\cW_{n,k}$ cover $\cV_{int}=\psi(U_{int})$. The same arguments that lead to \eqref{34} gives
\beq\label{eq t1}
\sum_{n\geq 1}   \sum_k (diam \,\cW_{n,k})^{t'} \geq  \frac{1}{B_* C^{t'}}\sum_{(m,l)\in \cI}    (diam \,\wh_{m,l})^{t '}
\eeq
where, this time,  $\cI$ is the set of all the $(n,k)$
such that $\wh_{m,l}\cap \cV_{int}\neq \emptyset$. 
Consequently,
$$
\pft \1 (w) \succeq   \frac{c_t\, x^{1-t}}{B_* C^{t'}}\sum_{(m,l)\in \cI} (diam \,\wh_{m,l})^{t'}\,.
$$
There exists $m_{min}$ such that for every $m\geq m_{min}$ the number  $\#\{l:(m,l)\in \cI\}$ is comparable to $4^m$ (\eqref{25.2}  of Lemma \ref{25}  this times applied with $\cV=\cV_{int}$). On the other hand,  $diam \,\wh_{m,l} \asymp \lm$
and,  by \eqref{35}, $\lm^\Th \geq C_f^{-1} 4^{-m}\frac{1}{m^\al}$. 
Since $t'<\Th$ it follows that $\pft \1 (w)$ is divergent.
\epf

\section{Approximation}\label{s7}

Up to now we have considered particular model functions and have obtained good estimates for their transfer operator.
But we really need  global entire functions having similar properties. Such functions will be obtained
with the help of an approximation result of model functions by entire functions. There are several approximation results,
the most general being the quasiconformal approximations by Bishop \cite{Bishop-S-2016, Bishop-EL-2015}.
We will use Rempe's uniform approximation \cite{Rempe-HypDim2} which is more restrictive but very precise.
It approximates models that are defined on the extended tract $\hat \Om \supset \Om$. Here is a version of his result.

\bthm[Uniform approximation]\label{thm:uniform_appr}
Let $\ph =\tau^{-1}=\hat \cH \to \hat \Om$ be a conformal map fixing infinity 
and normalized by \eqref{44}
and let $f=e^\tau : \hat \Om \to \C$.
Let $\cC:\R\to\hat\Om$ be defined by $\cC(t)= \ph ( it -13\log_+ |t| +1 )$.
Let $\tilde \Om$ be the component of $ \C\setminus \cC$ that is contained in $\hat \Omega$ and let $\tilde H = \tau (\tilde \Om )$.
So, $\overline{\tilde\cH}\subset \hat\cH$ and $\overline{\tilde\Omega}\subset \hat\Omega$.
Put 
$$
h(z)= \frac{1}{2\pi i}\int _{\mathcal C} \frac{f(\xi)}{\xi -z}d\xi \quad , \quad z\not\in \cC\,.
$$
Then this formula defines a holomorphic function for $z\notin \cC$ and the function 
$F$ defined as 
\begin{equation}
F(z)= \begin{cases}
f(z) + h(z) \quad \text{when}\quad z\in \tilde  \Om \;\; \text{and}\\
h(z) \quad \text{when}\quad z\not\in \overline{\tilde  \Om}
\end{cases}
\end{equation}
extends to an entire function $F$  in the class $\mathcal B$.
Moreover, the function $h$ satisfies the estimate
\begin{equation}\label{eq: estimating_h} 
|h(z)|\le \frac{C}{|z|_+}
\end{equation}
 where $C$ is some constant and where $|z|_+=\max(|z|,1)$.


\ethm

\subsection{Universality of estimates}
We shall use the above  approximation for varying  model functions $f$, and then pass from the estimates for the model to the estimates for the actual function $F$.  It is essential for further estimates to examine the error term $h$, i.e  the universality of the constant $C$ appearing in the inequality \eqref{eq: estimating_h}.
In order to check this universality it is sufficient to go carefully through very precise estimates provided in \cite{Rempe-HypDim2}.  

Indeed, the domain $\hat\cH$ is exactly the one considered in Remark 2 of Section 4 in \cite{Rempe-HypDim2}.  This domain is called ''initial configuration''. So, in the case under consideration  the ''initial configuration'' is fixed.



 
In  Corollary 4.5 in \cite{Rempe-HypDim2} the required estimate for the error function $h$ appears:
\beq\label{x.2'''} |h(z)|<M_5, \quad |h(z)|\le\max(|z_0|_+, \dist (z_0,\partial \hat\Omega))  \frac{M_6}{|z|_+} \eeq
for all $z\in \C\setminus \cC$. The  function $F$ is in class $\mathcal B$ since
\beq\label{x.25}S(F)\subset \overline{\mathbb D_{2M_5}}.\eeq
Here, the constants $M_5$ and $M_6$ depend only on the initial configuration, which is fixed. We may assume that $M_5\geq r\geq e^2$. The point $z_0$ which appears in \eqref{x.2'''}  is defined as 
$$z_0=\varphi(1).$$
It follows directly from the normal family property of the family of maps $\ph$ as explained in Remark \ref{normality f} that
 $\max(|z_0|_+, \dist (z_0,\partial \hat\Omega))\asymp 1$. Thus there exists $M_0 > 2\max \{M_5, M_6\}>r$
such that
 \beq \label{x.2}
|h(z)|\le \frac{M_0}{|z|_+}
\quad \text{ , }\; z \in \C\setminus C\eeq
 for all our examples.
In particular, we have the statement of Remark 4.6 in \cite{Rempe-HypDim2}: 
\beq\label{x.5}
\text{$|F|\leq M_0 \; $ outside $\; \tilde \Om\;$ 
and  $\;|F|\leq |f|+M_0\;$ in $\;\tilde \Om$. }\eeq

\smallskip

\subsubsection{Disjoint type and order}
The above estimates  allow us  to fix the translation constant $T$ in \eqref{10} such that all the models $\bf f$ and also the shifted 
approximation functions defined by $$\disF (z) = F(z -T)\quad , \quad z\in \C\; ,$$ are of disjoint type. 
The following lemma shows that this is the case whenever $T\geq 8M_0$ with $M_0$ from \eqref{x.2}. The precise choice of $T$, in fact of $\eta$
since we will set $T=4\eta$,  will be fixed in \eqref{x.6}.

\blem\label{x.4} Choose an arbitrary $\eta\geq 2M_0$ and set $T=4\eta$. Let 
$\bf f$ be any model of Section \ref{s2}. Then, every entire function $\bf F$ associated to $\bf f$ by the above construction is of finite order and
$$
\Om_{\bf g}:= {\bf g^{-1}} (\D_{\eta}^*) \subset \D_{2\eta}^*
$$
for $\bf g = f$ and for $\bf g=F$. Consequently,
$$ J_{\bf f}\subset  \D_{2\eta}^*
 \quad \text{and}\quad J_{\bf F}\subset  \D_{2\eta}^*.$$
\elem

\bpf
We first show that $\bf F$ is of finite order. Given the definition of the order in \eqref{order} and 
the estimate \eqref{x.5} it suffices to check that the model function $f$ is of finite order, i.e. that
$$
\limsup_{z\in \Om \; , \; |z|\to \infty} \frac{\log \log |f(z)|}{\log |z|} <\infty \,.
$$
But $f(z)=e^{\tau (z)}$ for $z\in \Om$. For $\tau =\ph ^{-1}$ we have the H\"older property 
\eqref{5'} which implies 
$ |\tau (z)| \leq \left( |z|/c_1\right)^K$, $ z\in \Om$, and thus $$\rho (f)=\rho (F) \leq K <\infty\,.$$

Let $\eta\geq  2M_0 $ and $T=4\eta$. Then, by construction of $\hat \Om$,
$$ {\bf f^{-1}} (\D_{\eta/2}^*)\subset \hat \Om + T = \ph (\hat \cH) +T \subset \D_{2\eta}^*$$ for all models $\bf f$.  
In particular $\Om_{\bf f}= {\bf f^{-1}} (\D_{\eta}^*)\subset \D_{2\eta}^*$.

Concerning $\bf F$,
if $z\in \Om_{\bf F}$ then, $|{\bf F }(z)| =|F(z-T)|> \eta \geq 2 M_0$ and thus the second inequality in \eqref{x.5} applies and gives
$$|{\bf f}(z)|= |f(z-T)|>\eta-M_0\geq \eta/2.$$
This shows that $z\in {\bf f^{-1}} (\D_{\eta/2}^*) \subset \D_{2\eta}^*$. The proof is complete.
\epf

\subsection{Comparing the transfer operators of the model function $f$ and of the approximating entire function $F$.}
 
Let $f:\hat\Omega\to\C $ be again one of our model functions and let  $F$ be the approximating entire map in class $\mathcal B$ produced by the  construction described in Theorem \ref{thm:uniform_appr}.

\begin{lem}\label{lem:approx}
There exists $R_0\geq 4M_0$ such that for all $z\in \Omega \cap \D_{R_0}^*$,
$$\frac{1}{2}\le \frac{|F(z)|}{|f(z)|}\le 2 \quad \text{and}\quad
\frac{1}{2}\le  \frac{|F'(z)|}{|f'(z)|}\le 2.$$
Consequently, $$\frac{1}{4}\le  \frac{|F'(z)|_1}{|f'(z)|_1}\le 4
\quad \text{,} \quad z\in \Omega \cap \D_{R_0}^*.$$
 Here, $R_0$ depends only on the constant $M_0$ from the estimate \eqref{x.2}.
\end{lem}

\begin{proof}
The estimate \eqref{x.2} implies
\beq\label{x.20}
 |h(z)|\le 1/4 \quad \text{for} \quad |z|\geq 4M_0 \; , \; z\not\in \cC.
 \eeq
Since $f:\Om=\Om_{\log r}\to \D_r^*$, for $z\in\Omega$ we have $|f(z)|>r\geq e^2$.
Consequently,
$$\frac{|F(z)|}{|f(z)|}=|1+h(z)/f(z)|\in \Big[\frac{15}{16},\frac{17}{16}\Big] \quad \text{,}\quad z\in \Om \cap \D_{4M_0}^*.$$

\noindent
Passing to the derivatives, if $\Phi:=h\circ \varphi$  then
$$F'=f'+h'=f'+\Phi '\circ \varphi^{-1}\circ (\varphi^{-1})'$$
and, as $f=\exp\circ \varphi^{-1}$,
$f'=f\cdot (\varphi^{-1})'$.
So,

$$\frac{|F'|}{|f'|}=\left |1+\frac{\Phi '\circ \varphi^{-1}}{f}\right | \quad \text{on}\quad \Om .$$
Since $|f(z)|>2$ in $\Omega$,  the required estimate relies on the estimate of $\Phi '$.
In order to estimate it, let $z\in\Omega$ and put $\xi=\varphi^{-1}(z)\in \cH_{\log r}\subset \cH_ 2$. 
Then $\D_\xi:=\D(\xi,1)\subset \cH_{1} \subset \ph^{-1} (\tilde \Om )$. This allows to make the estimate
$$
|\Phi '(\xi)|=\left |\frac{1}{2\pi i}\int _{\partial \D_\xi}\frac{\Phi(v)}{(v-\xi)^2}dv\right |\le  \sup_{\tilde z\in \varphi(\partial \D_\xi)}|h(z)|
$$
In order to use the estimate  \eqref{x.2}  for the function $h$ we need to estimate $\inf_{\tilde z\in\varphi(\partial \D_\xi)}|\tilde z|$. 
But this can be done by using twice the H\"older continuity property \eqref{5'}. It shows that 
$|\xi|\geq \left( |z|/c_2\right)^{1/K}$ and also that 
$$
|\ph (\hat z )|\geq c_1 |\hat z|^{1/K} \geq c_1 \Big(\big( |z|/c_2\big)^{1/K} -1\Big)^{1/K} 
\; \text{ for every } \; 
\hat z\in\partial \D_\xi.
$$
Choose now $R_0\geq 4M_0$ such that $c_1 \big((R_0/c_2)^{1/K} -1\big)^{1/K} \geq 4M_0$. Then, if $z\in \D_{R_0}^*$,
the corresponding $\inf_{\tilde z\in \varphi(\partial \D_\xi)}|\tilde z|\geq 4M_0$ which enables us to conclude using  \eqref{x.20}:
\beq\label{eq:est_K'}
|\Phi'(\xi)|\leq  \sup_{\tilde z\in \varphi(\partial \D_\xi)}|h(\tilde z)|\leq 1/4 \quad \text{for every $z\in \Om \cap \D_{R_0}^*$}
\eeq
which shows the required estimate of the ratio $|F'(z)|/|f'(z)|$. The estimate for the ratio $|F'(z)|_1/|f'(z)|_1$ follows directly, since
$|f'|_1(z)=|f'(z)|\cdot\frac{|z|}{|f(z)|}$ (and the analogous formula for $F$).
\end{proof}

Assume in the following that $R_0\geq 4M_0$ is such that Lemma \ref{lem:approx} holds and let
\beq \label{x.6}
\eta = 3 \max \Big\{M_0  , \exp\left( (R_0/c_1)^K\right)\Big\}.
\eeq 
Then  Lemma \ref{x.4} applies. Also
 \beq\label{x.7} f^{-1}(\D_{\eta / 3} ^*) \subset \D_{R_0}^*\subset \D_{4M_0}^*.\eeq
since $\exp ^{-1}\big( \D^*_{\eta /3}\big)\cap \D_{\log \eta /3} =\emptyset$
 and since \eqref{5'} implies $\ph ( \D_{\log \eta /3})\supset \D_{R_0}$.

\medskip

The transfer operator has been defined in \eqref{6}. Since we  now deal with several functions we write 
$\cL_{t,g}$ for the transfer operator of a function $g$.
We first compare the operators of an initial model $f$ and its approximation $F$.

\bprop\label{x.8}
There exists a constant $\cK\geq 1$ such that the following holds.
Let $f$ be a model as defined in Section \ref{s2}, $F$ an approximating entire function of $f$ 
given by Theorem \ref{thm:uniform_appr}. 
Then
$$
\frac1{\cK^t}\leq \frac{\pf_{t,  f}\1 (w)}{\pf_{t, F} \1 (w)}\leq \cK^t\quad \text{for all} \quad w\in \D_\eta ^*\,
$$
and the same holds if $f$, $F$ are replaced by their disjoint type versions ${\bf f}$, ${\bf F}$.
\eprop

We thus get first examples of entire functions for which the full thermodynamic formalism holds
in the particular case where $t$ equals the transition parameter $t=\Th$. 

\begin{cor}\label{x.8bis}
The transition parameter $\Theta$ is the same for the model $\bf f$ and for the approximating entire function $\bf F$.
Moreover, $\bf F$ is also of convergence type and Theorem \ref{32} as well as Remark \ref{32'}, hence the full thermodynamic formalism, meaning that all the results in Section 8 of \cite{MUpreprint3}, is also valid for the disjoint type entire function $\disF$ for all $t\geq \Th$.
\end{cor}

\bpf[Proof of Proposition~\ref{x.8}.]
Lemma \ref{lem:approx} shows that the values of the derivatives of $f$ and $F$ are comparable at a given point $z$.
But, in the formulas defining the operators  $\pftf$ and $\pftF$  the summation runs over preimages of a given point $w$ under $f$ and $F$, respectively. So, in order to compare $\pftf(w)$ and $\pftF(w)$, the preimages of $w$ under $f$ and $F$ will be ''paired'' and the derivatives of $f$ and $F$ on these paired preimages will be compared.

Let $w\in \D_\eta ^*$. Then all preimages of $w$ under the model map $f$ are in $\Om_{\log \eta}=\varphi(\mathcal H_{\log \eta})$ and 
$$\sum_{z\in f^{-1}(w)} |f'(z)|^{-t}=\sum_{\xi\in \exp^{-1}(w)}\left (\frac{|\varphi'(\xi)|}{|\varphi(\xi)|}\right )^t.$$
Take the circle $\sg $ centered at $w$, with radius $1$, and 
for each  $\xi\in\exp^{-1}(w)$  let  $\gamma_\xi$ be the preimage of $\sg$ under $\exp$, surrounding $\xi$.   Finally, put $\tilde\gamma_\xi=\varphi(\gamma_\xi)$. Notice that the domain bounded by $\tilde\gamma_\xi$ contains exactly one preimage of $w$ under the map $f$; this is the point $z=\varphi(\xi)$.

On each curve $\tilde\gamma_\xi$ we have that $|f(z)-w|=1$, while
\beq\label{x.23}|(F(z)-w)-(f(z)-w)|=|h(z)|<\frac{M_0}{|z|}\eeq
where $M_0$ comes from \eqref{x.2}.
From \eqref{x.7} we know that $|z|>4M_0$ since $w\in\D_\eta ^*\subset \D^*_{\eta /3}$.  Hence, the right hand inequality of \eqref{x.23} is strictly less than $1$.
This  allows to conclude via Rouch\'e's Theorem that $F$ has exactly one preimage in the region bounded by $\tilde\gamma_\xi$. Denoting this preimage by $\tilde z$, we need to compare 
$|f'(z)|_1$ and $|F'(\tilde z)|_1$.
But this directly follows from Koebe's Distortion Theorem and Lemma~\ref{lem:approx}, and the constant $\cK$ in Proposition~\ref{x.8} is exactly 
a Koebe constant times an absolute one.
This gives the first part of the required estimate, i.e.
$$\frac{\pf_{t,  f}\1 (w)}{\pf_{t, F} \1 (w)}\leq \cK^t$$
The second  part of the estimate can be obtained in a similar way: Let $w\in \D_\eta^*$. 
Since $\eta\geq2M_0>M_0+2$, the disk $\mathbb D(w,2)$ does not contain singular values of  
$F$, and  $F^{-1}(\mathbb D(w,1))$ is a countable union of Jordan domains $\mathcal D_I$
each of them being mapped bijectively and with bounded distortion onto $\mathbb D(w,1)$. 

If $z\in \mathcal D_I$ then $|F(z)| \geq |w|-1 > \eta -1\geq 2M_0$. It thus follows from \eqref{x.5} that $z\in \tilde \Om$
and thus all the domains $\mathcal D_I\subset  \tilde \Om$. Moreover, still using \eqref{x.5},
$$
|f(z)|\geq |F(z)|-M_0 > \eta -1 -M_0\geq \frac 23 \eta -1 \geq \frac \eta 2
$$
since $\eta \geq 3M_0\geq 6$. This allows to apply \eqref{x.7} and thus to get $|z|\geq 4M_0$.

On the curve $\gamma_I$ bounding $\mathcal D_I$ we have $|F(z)-w|=1$, while 
$$|(F(z)-w)-(f(z)-w)|=|h(z)|<\frac{M_0}{|z|_+}\leq \frac{M_0}{4M_0}=1/4\, .$$  Again, Rouche's theorem implies that $f$ has 
exactly one preimage of $w$ in each domain $\mathcal D_I$.  Applying again Koebe's Distortion Theorem and Lemma~\ref{lem:approx}, we obtain the desired inequality:
$$\frac{\pf_{t,  f}\1 (w)}{\pf_{t, F} \1 (w)}\geq \frac{1}{\cK^t}.$$

\smallskip

Let us finally consider the disjoint type functions $\bf f, F$. If $z, \tilde z$ is a
pair of preimages of $w$ under $ f, F$ then, clearly, ${\bf z}= z+T,\,  {\bf \tilde z}= \tilde z+T$ is a pair of preimages
of $\bf f$ and $\bf F$ respectively and we have
$$
\frac{|{\bf f'}({\bf z})|_1}{|{\bf F'}({\bf \tilde z})|_1} = \frac{|{ f'}(z)|_1}{|{ F'}(\tilde z )|_1} \left| \frac{\bf z}{{\bf z}-T}\frac{{\bf\tilde z} -T}{{\bf \tilde z}}\right|
\asymp  \frac{|{ f'}(z)|_1}{|{ F'}(\tilde z )|_1}
$$
with involved multiplicative constants independent of the functions, of the point $w\in \D_\eta^*$ and of the 
pair of preimages. 
This clearly completes the proof of Proposition \ref{x.8}.
\epf

There is also a relation between the transfer operator of the functions and their disjoint type version.

\blem\label{x.10} Let $A= 1+\frac T4$. Then
$$
\frac{1}{A^t} \pftf \1 \leq \cL_{t, {\bf f}}\1\leq A^t \pftf \1 \quad \text{on}\quad \D_\eta^* .
$$
\elem

\bpf
This follows from an elementary estimation based on \eqref{x.9} and on
\beq\label{x.21'}
|{\bf f'} ({\bf z})|_1 = |{ f'} ({\bf z}-T)|_1\frac{|{\bf z}|}{|{\bf z}-T|} \;\, , \quad {\bf z}\in \Om_{\bf f} = {\bf f}^{-1}(\D_\eta^*).
\eeq
\epf

\section{Topological pressure and Bowen's Formula}\label{s8}

Let $f$ be a disjoint type model or entire function and consider again $\pft =\pftf$ its tranfer operator.
 By Theorem 8.1 of \cite{MUpreprint3} the limit
\beq\label{36} P(t )= P_f(t)=\lim_{n\to\infty} \frac{1}{n} \log \pf^n_t\1(w)\eeq
exists and, by bounded distortion, it 
does not depend on $w\in \D_r^*$ (for $r$ sufficiently large). This limit is called \emph{topological pressure} and for a convergence type function the pressure $P(\Th ) $ is finite.
The basic properties  is that $t\mapsto P(t)$ is a convex and strictly decreasing function
on $(\Theta , \infty )$ with $P(t)=\infty$ if $t<\Th$, $P(t)$ is finite if $t>\Theta$ and $\lim_{t\to\infty}P(t)=-\infty$.  
Consequently, the map $t\mapsto P(t)$ has a unique zero $h>\Th$ provided there exists $t>\Th$ such that $P(t)>0$.

We refer to \cite{ Mayer:2017ab} for the notion of H\"older tract. All what is needed here is that the
 tracts of our examples have this property since they are quasidisks.

\bprop \label{38}
Assume that the disjoint type entire function $f$ has only one logarithmic  tract, assume that this tract is H\"older.
Then
$$
HypDim(f)=\inf \big\{ t>0\; , \;\; P(t)<0 \big\}.
$$
\eprop

\bpf
Consider first the case that $P(t)>0$ for some $t>\Th$ in which case the pressure has a unique zero $h>\Th$. 
The assumptions on $f$ imply that  \cite{MUpreprint3} applies to them and, in this case, the statement in  Proposition \ref{38} is exactly 
the Bowen's Formula in  \cite{MUpreprint3} which states that $$HypDim(f) = h > \Th.$$

\noindent
It remains to consider the case where $P(t)\leq 0$ for $t>\Th\;$ and, clearly, $\;P(t) = \infty$ for $t<\Th$.
We then have to show that $$HypDim(f) = \Th.$$
The H\"older tract assumption along with \cite{Mayer:2017ab} gives $HypDim(f) \geq \Th$ no matter how $P$ behaves.
For the other inequality, let us first recall that the thermodynamic formalism of  \cite{MUpreprint3} applies to $f$ for every parameter $t> \Th$. In particular, there exists $e^{P(t)}$--conformal measure which allows to employ Lemma 8.1  in \cite{MyUrb08}. This Lemma gives the required estimate since it shows that  $HypDim (f)\leq t$ whenever $P(t)\leq 0$.
\epf


\section{Convergence type entire functions with positive pressure}\label{s6}

For the models of the previous section the topological pressure, introduced in \eqref{36}, is finite for every $t\geq \Th$
but certainly we may have $P(\Th ) <0$. Here we consider the disjoint type versions of the models given by Example \ref{exam 2}
and show that they have positive pressure for $t=\Th$ and even for slightly larger values of $t$
provided the number $N$ in Example \ref{exam 2} has been chosen sufficiently large.
We then also show that this property is true for the disjoint type approximating entire functions.

We recall that the models of Example \ref{exam 2} are special cases of those of Example \ref{exam 1}.
Therefore,  they are of convergence type with $\Th =\log 4 /\log 3$ and 
Theorem \ref{32} applies.

\bprop\label{37} Let $f$ be a model of Example \ref{exam 2} and 
let  $\Th =\log 4 /\log 3$.
Then, for sufficiently large $N$ there exists $t>\Th$ such that 
$$P_g(t)>0 $$
where 
$g=\disf$, the disjoint type version of $f$, and also if $g=\disF$, the disjoint type version
of the entire function $F$ approximating $f$.
\eprop

\bpf
First,  we establish an auxiliary  estimate  for the initial model function $f$.
We shall prove that, choosing sufficiently large $N$ in the model in Example \ref{exam 2},  one can 
 find $S>\eta$ such that 
\beq\label{45}
\pf_{t,f}(\1_{\D_{S'}} )(w) \geq 2A^t
 \quad \text{for every} \quad \eta\leq |w| \leq S
\eeq
and for some $t>\Th$, where $S'=S-T$
and with $A$ from Lemma~\ref{x.10}.


In order to establish  \eqref{45}, let $N\geq 1$ be maximal such that $2^N\leq S$, and let $M$ be determined by  the inequality $2^{M-1}<T\le 2^M$.  
Consider any $w\in \D_S\cap \D_\eta ^*$ and set $x=\log |w|\in [\log \eta, \log S]$
where $\eta$ is given by \eqref{x.6}.
Let  $j\geq 0$ be again the maximal integer such that $\mu^{j-1} \leq x$. Notice that 
\beq\label{eq S} j\preceq \log (\log S) \asymp \log N.\eeq

We have to estimate $\pft (\1_{\D_{S'}} )(w) $ and, in order to do so, we first describe the preimages 
$z\in f^{-1} (w)$ that are in the disk $\D_{S'}$. We have $z=\ph(\xi)=\psi(h(\xi))$ where $\xi =x+iy$ and where $y\in I_{n+j}$ for some 
$n\geq 0$. Selfsimilarity of $\psi$ (Lemma \ref{4}) yields
$$
z=2^{n+j}\psi \left(\mu^{-(n+j)}h(\xi )\right).
$$
On the other hand, $
|h(\xi )|\asymp |\xi| \asymp \mu^{n+j}
$ since $h'(\infty)=1$ (see Proposition \ref{41}) and thus $|\psi  \left(\mu^{-(n+j)}h(\xi )\right)|\preceq 1$
and $|z|\preceq 2^{n+j}$. Denote by $c$ the constant in the last inequality, meaning that it becomes
$|z|\leq c 2^{n+j}$. Then, 
by the choice of $N$, we see that $z\in \D_{S'}$ if
$
c2^{n+j}< 2^N-2^M
$. This is the case if $n< N-j +\left (\log(1-2^{M-N})- \log c\right ) / \log 2$ and thus, because of \eqref{eq S} and since we will take $N$ large, it thus suffices to have 
$n\leq N/2$.

Given this discussion on the preimages of $w$, we get out of  the expression of $\pft$ in Proposition \ref{13}
that
\begin{align*}
\pf_{t,f} (\1_{\D_{S'}} )(w) &\succeq x^{1-t} \sum_{n=1}^{[N/2]}\left\{ \left(1+\frac{n+j}{x}\right)^{1-t} \sum_k(diam \,\cW_{n,k})^t \right\}\\
&\succeq N^{1-t} \sum_{n=1}^{[N/2]}\left\{ \left(1+N\right)^{1-t} \sum_k(diam \,\cW_{n,k})^t \right\}\\
&\succeq N^{2(1-t)} \sum_{n=1}^{[N/2]} \sum_k(diam \,\cW_{n,k})^t.
\end{align*}
The sets  $\cW_{n,k}$ can now be replaced by  the covering $(\wh_{m,l})$ precisely like we did in the proof of Theorem \ref{32}. 
More precisely, we use  \eqref{eq t1} with the difference that we deal here with a finite sum:
$$
B_* C^t \sum_{n=1}^{[N/2]}\sum_k (diam\, \cW_{n,k})^t \geq  \sum_{(m,l)\in \cI_{finite}} (diam\, \wh_{m,l})^t
$$
and we must specify the new set of indices $\cI_{finite}$ over which the summation goes. 
In order to do so, we recall first that the sets $\cW_{n,k}$ cover $\cV_{int}=\psi (\cU_{int})$. Therefore, if
$(m,l)$ is such that  $\wh_{m,l}\cap \cV_{int}\neq \emptyset$ then there exists $(n,k)$ such that  $\wh_{m,l}\cap \cW_{n,k}\neq \emptyset $  and then, by Lemma \ref{31}, $n\leq m/b$. We can thus take
$$
\cI_{finite}
=\{(m,l)\; , \; \wh_{m,l}\cap \cV_{int}\neq \emptyset \, \text{ and } \, m_0\leq m\leq b[N/2]\}$$
where $m_0$ comes from Lemma \ref{25}.

By \eqref{25.2} of Lemma \ref{25}, for every $m\geq m_0$ the number of 
indices $(m,l)$ in $\cI_{finite}$ is comparable to $4^m$. Also, $diam\, \wh_{m,l}\asymp \lm$ and for the models of Example \ref{exam 2}  we have $\lm=\big(\frac 1{e}\big) ^m$ if $1\leq m\leq N$. 
Consequently, if $N$ is large enough, we get all in all
\beq\label{zzz}
\pf_{t,f} (\1_{\D_{S'}} )(w) \succeq N^{2(1-t)}  \sum_{m=m_{0}}^{[b[N/2]]} 4^m \left(\frac 1{e}\right) ^{mt}
\asymp   \; N^{2(1-t)}   \left(\frac 4{e^t}\right) ^{bN/2}  \,
\eeq
which is arbitrarily large provided we take 
$\Th < t < \log 4 / \log e = \log 4$ and provided that $N$ is sufficiently large.

\medskip

Coming now to the associated disjoint type model $\bf f$, and using Lemma~\ref{x.10}
we can translate the estimate \eqref{45} to the  case of $\bf f$ as follows:
\begin{equation}\label{45a}
\pf_{t,\bf f}(\1_{\D_{S}})(w) \geq 2
 \quad \text{for every} \quad \eta\leq |w| \leq S
\end{equation}
Indeed, if $w\in \mathbb D^*_\eta\cap \mathbb D_S$ and if $f(z)=w$ then 
${\bf f} ({\bf z})=w$ where ${\bf z} = z+T$. 
Moreover, if $z\in \D_{S'}$ then ${\bf z} \in \D_S$.
Combining now \eqref{45} with  the estimate in Lemma ~\ref{x.10} we obtain directly the 
required  \eqref{45a}.

But now, since $\bf f$ is of disjoint type, and, in particular ${\bf f}^{-1}(\mathbb D_\eta^*)\subset \mathbb D^*_\eta$, 
\eqref{45a} allows us to conclude inductively:
$$
\pf_{t, {\bf f}}^n\1 (w)\geq \pf_{t, {\bf f}}(\1_{\D_S}\pf_{t, {\bf f}}(\1_{\D_S} ...  \pf_{t, {\bf f}}(\1_{\D_S} )) )(w) \geq 2^n
\quad \text{for every} \quad n\geq 1.
$$
Therefore,  $P_{\bf f}(t)>0$.


\medskip

It remains to verify that the entire function $\disF$ also has positive pressure at $t$.
Proposition \ref{x.8} compares the operators of $\disf$ and $\disF$ but
with transfer operators applied to the constant function $\1$ and we have to replace it  by $\1_{\D_S}$. 
So let $w\in \D_S\cap \D_\eta^*$, consider a pair of preimages $z,\tilde z$
of $w$ under  $f,F$ respectively defined exactly like in the proof of Proposition \ref{x.8}.
Then ${\bf z}=z+T, {\bf \tilde z}=\tilde z+T$ are corresponding preimages of $w$ under $\disf , \disF$ respectively.
It is explained in this proof that, given $z\in f ^{-1} (w)$, there exists a unique $\tilde z
=\tilde z (z) \in F^{-1} (w)$ which is in the region bounded by $\ph (\ga _\xi) $.
An elementary estimation shows that $diam (\ga _\xi )\leq \frac 2 \eta$. Since $\ph$ fixes the origin and is uniformly quasisymmetric, it follows that there exists a constant $\tilde K\geq 1$ such that
$$
\tilde \ga _\xi =\ph (\ga _\xi) \subset \ph (\D_{|\xi|+\frac 2 \eta}) \subset \D_{\tilde K |\ph (\xi )|}= \D_{\tilde K |z|}.
$$
If again $S'=S-T$ then 
$$F^{-1}(w)\cap \D_{S'} \supset \{\tilde z = \tilde z (z) \; ,\;\;|z| < S'/\tilde K \;\; \text{and} \;\; f (z) =w\}.$$
Since Lemma \ref{x.10}, in fact \eqref{x.21'}, is also valid for $F,\disF$ instead of $f, \disf$, we get 
\begin{align*}
\pf_{t,\disF }(\1_{\D_S})(w) \geq & A^{-t} \pf_{t,F }(\1_{\D_{S'}})(w) \\
\geq &  A^{-t}  \sum_{ \substack{\tilde z (z)  ,\;  f (z)=w\\ \;|z| < S'/\tilde K\;}} |F' (\tilde z)|_1^{-t}
 \succeq  \pf _{t, f}(\1_{\D_{S'/\tilde K}})(w)  \; \,
 \text{, $\; w\in \D_S\setminus \D_\eta^*$,}
\end{align*}
 the last inequality resulting from the
proof of Proposition \ref{x.8}. In conclusion, in order to get \eqref{45} for the function $\disF$
it suffices to adjust the number $N$ so large such that $\pf _{t,f}(\1_{\D_{S'/\tilde K}})$
is sufficiently large on $\D_S \cap \D_\eta^*$ which is possible because of \eqref{zzz}.
\epf


\section{Proof of Theorem \ref{theo main}}

Let $f$ be a model such that the associated disjoint type entire function $\disF$ has positive pressure (Proposition \ref{37}).
Consider the analytic family of entire functions:
$$\disF_\l =\l \,\disF  \quad , \quad \l \in \C^*\,.$$ 

\noindent

\bprop\label{x.24}
The functions $\disF_\l $, $0<\l \leq 1$ do all belong to the same hyperbolic component of the parameter space of $(\disF_\l)_\l$.
\eprop

\bpf
By Lemma \ref{x.4} the tract $\Om_1$ of $\disF_1$ satisfies $\Om_1= \disF_1^{-1}(\D_\eta^*) \subset \D_{2\eta}^*$.
Clearly, for every $\l \in \ov \D\setminus \{0\}$, $\Om_\l = \disF_\l^{-1}(\D_\eta^*) \subset \Om_1$ and thus
$\Om_\l \subset  \D_{2\eta}^*$. Therefore, all the functions $\disF_\l$,  $\l \in \ov \D\setminus \{0\}$, are of disjoint type 
and thus hyperbolic.

It remains to find a simply connected domain $V\subset \C\setminus \{0\}$
that contains $(0,1]$ along with a holomorphic motion $(\ph_\l )_\l$, $\l\in V$, that
identifies the Julia sets and conjugates the dynamics of $\disF_1$ and $\disF_\l$.
But this has been shown  in Section 3 of the paper \cite{Rempe09} by Rempe.
\epf


\bprop\label{x.30}
There exists $0<l_0 <1$ such that 
$$P_{\disF_\l} (\Th ) <0$$
for every $0<|\l | \leq l_0$.
\eprop

\bpf
Again by Lemma \ref{x.4}, $\disF_\l ^{-1} (\D_{\eta}^*) \subset \D_{2\eta}^*$,  $\l \in \ov\D \setminus \{ 0\}$. In particular, $\J_{\disF_\l}\subset \D_{2\eta}^*$
for all these parameters and it suffices to study the transfer operator on $\D_{\eta}^*$.

Notice that  $\pf_{t, \disF_\l}\1(w)= \pf_{t, \disF}\1(w/\l)$ for every $w\in \D_{\eta}^*$ where $\disF=\disF_1$.
On the other hand, Proposition \ref{x.8} and Lemma \ref{x.10} imply for the operator of the generating function $\disF=\disF_1$
$$
 \pf_{t, \disF}\1\leq \cK^t  \pf_{t, \disf}\1\leq (A\cK )^t \pf_{t, f}\1
$$
still on  $\D_{\eta}^*$.
Moreover, we have  Theorem \ref{32} which implies, for every $t\geq \Th$,
$$
 \pf_{t, f}\1 (w) \leq M^t (\log |w|)^{1-t}\quad , \quad w\in \D_{\eta}^*.
$$
Combining all these relations and taking $t=\Th$ we get
$$
\pf_{\Th , \disF_\l} \1 (w) = \pf_{\Th , \disF}\1(w/\l) \leq (A\cK M) ^\Th \Big(\log (\eta /l_0)\Big)^{1-\Th }
$$
for every $w\in \D_{\eta}^*$ and every $0< |\l | \leq l_0$. 
 Since $t=\Th =\log 4 / \log 3 >1$ we can choose $l_0$
small enough so that
$$ (A\cK M) ^\Th\Big(\log (\eta /l_0)\Big)^{1-\Th } \leq \frac 12.$$
Then 
$$
\frac 1n \log \pf^n_{\Th , \disF_\l} \1 (w) \leq \log 1/2 \quad \text{for every $w\in \D_{\eta}^*$ and every $n\geq 1$}
$$
which implies that $P_{\disF_\l} (\Th ) <0$ whenever $0< |\l |\leq l_0$.

\epf

\bpf[Proof of Theorem \ref{theo main}]
Given  Proposition \ref{x.24} and the fact that $\bf F$ is of finite order (Lemma \ref{x.4}), it remains to show that the hyperbolic dimension does not vary analytically.
We know from Proposition \ref{37} that $P_{\disF_1}(t )>0$ for some $t>\Th =\log 4 / \log 3$. In this case,
the Bowen's Formula in  Proposition \ref{38} shows that
$$HypDim(\disF_1 ) > \Th\,.$$
On the other hand, $P_{\disF_{\l}}(t )<0$ for all $\l \in (0,l_0]$ where $l_0$ comes from Proposition \ref{x.30}.
Again  Proposition  \ref{38} shows then that 
$$HypDim(\disF_\l) =\Th \quad\text{for every}\quad  0<\l \leq l_0.$$
Consequently, $\l\mapsto HypDim(\disF_\l )$ is not  an analytic function.
\epf

 \smallskip

\section{Irregular hyperbolic functions in Class $\cB$}\label{s9}

In this section we proof Theorem \ref{theo main''}, \ref{theo main'''}
and \ref{theo main'}. First of all, all our examples share the particular value $\Th =\log 4 /\log 3$. But clearly the snowflake construction can 
be modified in order to get functions with the same behavior and with $\Th$ any value in $ ]1,2[$.
The only modification is the choice of the numbers $\rho_n \in [\rho_{min} , \rho_{max}]$ where then
$\rho_{min}, \rho_{max}$ have to be fixed such that \eqref{1} is replaced by
$$
\frac 14 < \rho_{min} < \Big( \frac12\Big)^{2/\Th}< \rho_{max} <\frac 12 \,.
$$
So, we can restrict the discussion here to the particular value  $\Th =\log 4 /\log 3$. 

We know from Lemma \ref{x.4} that all the entire functions we consider 
are of finite order.
 From Proposition \ref{x.30} we directly get functions that fulfill the requirements of Theorem \ref{theo main''}. Combining it with 
the Bowen's Formula of Proposition \ref{38}, Theorem \ref{theo main'''} also follows. 
The remaining point is to show
the affirmation concerning the conformal measure in Theorem \ref{theo main'}.

\smallskip

In view of establishing it
 we need some preliminary considerations on the choice of the Riemannian metric and to clarify the notion
 of conformal measure.
Up to now we have used the cylindrical metric in order to evaluate the derivatives (see \eqref{6'}).
This choice is related to the logarithmic coordinates in \cite{EL92} and it allows to get
a bounded transfer operator as defined in \eqref{6}. 
 However, it is sometimes more convenient 
to make a different choice. For example, employing  the spherical metric allowed the authors in \cite{BKZ-Bowen}
to get the most general Bowen's Formula. 

Consider a general Riemannian metric $d\rho (z)=\rho(z)|dz|$
on $\C$, denoted by $|f'|_\rho= |f'|\frac{\rho\,  \circ f}{\rho}$ the derivative with respect to it 
 and let us have in mind the particular choices
$$
d\rho_{cyl} (z)=\frac{|dz| }{|z|} \quad \text{and} \quad d\rho_{sph} (z)=\frac{|dz| }{1+|z|^2}.
$$
The cylindrical metric as written has a singularity at the origin, a problem that we can neglect since 
we work far away from it especially in the case of disjoint type functions.

\bdfn\label{conf mea} Let $f$ be an entire function.
A finite measure $\nu$ is said to be  $t$--conformal with respect to the metric $\rho$  if for every Borel set $A\subset \mathbb C$ such that $f_{|A}$ is 
injective we have 
$$\nu(f(A))=\int_A |f'|^t_\rho d\nu.$$
\edfn

\noindent
As defined, such a measure is sometimes also called geometric conformal measure since  such a measure is commonly used to analyse the geometry of the Julia set. 

The topological pressure with respect to the cylindrical metric has been defined in \eqref{36}. If $\cL_{\rho,t}$ denotes the operator defined by Formula \eqref{6} but with $|f'|_1$ replaced by $|f'|_\rho$ and if we inject this operator in \eqref{36}
then this defines the  topological pressure 
with respect to the metric $\rho$: 
$$
P_\rho(t) =P_{\rho,f}(t)= \lim _{n\to \infty}\frac 1n \log \cL^n_{\rho,t}\1 (w) \;\; , \;\; w\in\ov  \D_r ^*.
$$
A priori, the transition parameter $\Th_\rho = \inf \{ t>0 \; , \; P_\rho (t) <\infty\}$ can depend on the metric.
In the case of the cylindrical or spherical metric we also write $P_{cyl}$, $\Th_{cyl}$  respectively $P_{sph}$, $\Th_{sph}$.
Recall that for our examples  $\Th_{cyl}=\log 4/\log 3$ and, right from the definition of the pressures, it is clear that
$$
P_{sph}(t) \leq P_{cyl}(t)
$$
hence $\Th_{sph}\leq \Th_{cyl}$. Given these notations, we can now show the following result which contains Theorem \ref{theo main'}.

\bthm
For every $1<\Th<2$ there exists a disjoint type entire function of finite order $f\in \cB$ with transition parameter $\Th$, with $HypDim (f)=\Th $ 
 and which does not have a spherical nor cylindrical conformal measure
supported on its radial Julia set.
\ethm

\bpf Again, we treat the case $\Th =\Th_{cyl} =\log 4 /\log 3$.
Let $f=\disF_{\l_0}$ be the disjoint type entire function of finite order from Proposition \ref{x.30}.
This function has negative cylindrical pressure at $\Th_{cyl}$ and thus
\beq\label{cyl sph}
P_{sph}(\Th_{cyl}) \leq P_{cyl}(\Th_{cyl})<0 .
\eeq
Bowen's Formula (Proposition \ref{38}) implies then that $\Th_{cyl}=HypDim(f)$. We also dispose in 
the same Bowen's Formula with respect to the spherical metric (\cite{BKZ-2009}) so that
$$\Th_{cyl}=HypDim(f)=\inf \{t>0 \; , \;\; P_{sph} (t)<0\}.$$
Combined with \eqref{cyl sph} and with the continuity of $t\mapsto P_{sph} (t)$ on $]\Th_{sph} , \infty [$
we get that 
$$\Th_{cyl}=\Th_{sph}$$
and thus $P_{sph}(t)<0$ for $t\geq \Th_{sph}$ and $P_{sph}(t)=\infty$ if $t< \Th_{sph}$.

Now, assume that this map $f$ has a spherical $t$--conformal measure  supported on $J_r(f)$ for some $t>0$.
Then necessarily $t\ge \Th_{sph}$ and $P_{sph}(t)=0$ by Theorem A in \cite{BKZ18}. But this is not possible 
as we have seen just above and thus such a conformal measure cannot exist.

The analogue for the cylindrical conformal measure also follows. Indeed, assume that $\nu$
is a cylindrical $t$--conformal measure  supported on $J_r(f)$ for some $t>0$. Then
$$
dm= \left( \frac{|z|}{1+|z|^2}\right)^t d\nu
$$
would define a finite spherical $t$--conformal measure  supported on $J_r(f)$. 
But such a measure cannot  exist if   $P_{sph}(t)<0$ (see  Proposition 3.3 in \cite{BKZ18}).
\epf

\section{Appendix}\label{appendix} 
Throughout the paper we used good bilipschitz properties of $h$ and of the rescaled functions $h_m=\mu^m\circ h \circ \mu^{-m}$. They follow from the fact that $h'$ has continuous extension to the boundary and this follows from the smoothness of the boundary of $\hat \cH$. Indeed, the relation between continuous extension of the derivative of a conformal map to the boundary and the geometry of the boundary is the object of Section 3 in Pommerenke's book \cite{PommerenkeBook}. The relevant fact for our application is that the derivative of a conformal map from the unit disk $\D$ onto the inner domain of a Jordan curve $C\subset \C$ has continuous extension to the boundary if $C$ is  Dini-smooth
(see Theorem 3.5 in \cite{PommerenkeBook}). This means that $C$ admits a parametrization $\al :\mathbb S^1=\{|z|=1\}\to C$ whose derivative $\al'$ is Dini-continuous:
$$
\int _0^\pi t^{-1} \om (t, \al ', \mathbb S^1)\, dt < \infty
$$
where the modulus of continuity $\om$ of $\al'$ on a set $A$ is defined by
$$
\om(t,\al ', A) =\sup \Big\{ |\al' (\xi_1 ) -\al'(\xi _2 ) | \; , \;\; |\xi_1-\xi_2|\leq t\; , \; \xi_1,\xi_2\in A \Big\} \,.
$$
The domain $\hat\cH$ and a boundary parametrization $\ga$ has been defined in \eqref{39}. 
In fact, $\partial \hat\cH= \{\sg (y) +iy \; , \; y\in \R\}$.
Since $\sg$ is $C^\infty$--smooth we only have to check what happens near infinity. In order to do so,
consider $\al : I=[-1/2, 1/2] \to \R$ defined by $\al (0)=0$ and 
$$
\al (t) = \frac{1}{\sg (1/t)+ i/t} \quad , \quad 0 < |t| \leq 1/2\,.
$$
\blem\label{40}
The domain $\hat \cH$ is Dini-smooth.
\elem

\bpf
The function $\al \in C^1$ with $\al '(0)=-i$ and 
$$
\al'(t)=\frac{i-14t}{(14t\log |t| -7t +i)^2} \quad , \quad 0< |t| \leq 1/2\,.
$$
Given this derivative, a direct calculation gives for the modulus of continuity 
$\om (t,\al ' ,I)=O (t\log 1/t)$ which shows that $\int _0^{1/2} \frac{\om (\al ' , t, I)}{t} dt <\infty$.
\epf

Theorem 3.5 in \cite{PommerenkeBook} therefore applies and gives  that the derivative of $\tilde h $
defined by $\tilde h (z) = 1/ h(1/z)$ has continuous extension to the boundary of the inverse of the domain $\hat \cH$.
In particular $\tilde h'(0)$ exists and in fact $\tilde h'(0)=1$ because this corresponds to the normalization
$h'(\infty ) =1$ that we assumed in Section \ref{s2}.

Remember that we introduced the rescaled maps 
$$h_m=\mu^{-m}\circ h \circ \mu^m : \hat\cH _m = \mu^{-m}  \,\hat\cH  \to \cH$$
in Section \ref{section 4.3}.
\bprop\label{41}
$|h'| \asymp 1$ and $|h_m'| \asymp 1$ uniformly in $m$
and
$h:\hat \cH \to \cH$ and the maps  $h_m: \hat\cH _m \to \cH$ are uniformly bilipschitz.
Moreover, when restricted to $\hat \cH_m\cap \{|z|\geq \mu^{-2}\}$, then the bilipschitz constant $L_m$
of the maps $h_l$, $l\geq m$, satisfies $L_m\to 1$ as $m\to\infty$. Finally,
$$h_m\longrightarrow Id_{\cH} \;\;\;\; \text{as}\;\; m\to\infty.$$
\eprop

\bpf 
The assertion on the derivatives holds since we checked that the domain $\hat\cH$ is Dini-smooth  (Lemma \ref{40})
which then allows to apply Theorem 3.5 in \cite{PommerenkeBook}. 
From this we also get the bilipschitz property since the domains $\cH$ and $\hat \cH_m$ have
sufficiently good convexity properties and $L_m\to 1$ results from $h'(\infty)=1$.

Concerning the last statement, consider $h_m^{-1}:\cH\to \hat \cH_m$ and let $g=\lim_{j\to\infty}h_{m_j}^{-1}:\cH\to \ov \cH$ be the limit of a convergent subsequence. Then $|g'|=1$ in $\cH$ and so $g$ is non-constant, hence a conformal self map of $\cH$. Again since $|g'|=1$ in $\cH$ and since $h_m(0)=0$ for every $m\geq 1$, $g$ is the identity map.
\epf

\blem\label{BL appendix}
$$
\frac{|y|}{L_m}\leq |\Im (h_m (r+iy ))| \leq L_m| y| \quad \text{for all}\quad y\in \R \text{ and } r>0.
$$
\elem

\bpf
Remember that $\ov h (z)=h(\ov z)$, $z\in \cH$. This symmetry implies that $h([0,\infty))=[0,\infty)$ and thus Lemma \ref{BL appendix} follows directly from the fact that $h_m$ is $L_m$--bilipschitz.
\epf


\medskip


\bibliographystyle{plain}

\begin{thebibliography}{10}

\bibitem{Ahl06}
Lars~V. Ahlfors.
\newblock {\em Lectures on quasiconformal mappings}, volume~38 of {\em
  University Lecture Series}.
\newblock American Mathematical Society, Providence, RI, second edition, 2006.
\newblock With supplemental chapters by C. J. Earle, I. Kra, M. Shishikura and
  J. H. Hubbard.

\bibitem{AR97}
James~W. Anderson and Andr\'{e}~C. Rocha.
\newblock Analyticity of {H}ausdorff dimension of limit sets of {K}leinian
  groups.
\newblock {\em Ann. Acad. Sci. Fenn. Math.}, 22(2):349--364, 1997.

\bibitem{AstZins94}
K.~Astala and M.~Zinsmeister.
\newblock Holomorphic families of quasi-{F}uchsian groups.
\newblock {\em Ergodic Theory Dynam. Systems}, 14(2):207--212, 1994.

\bibitem{AIM-2009}
Kari Astala, Tadeusz Iwaniec, and Gaven Martin.
\newblock {\em Elliptic partial differential equations and quasiconformal
  mappings in the plane}, volume~48 of {\em Princeton Mathematical Series}.
\newblock Princeton University Press, Princeton, NJ, 2009.

\bibitem{B07}
Krzysztof Bara\'nski.
\newblock Trees and hairs for some hyperbolic entire maps of finite order.
\newblock {\em Math. Z.}, 257(1):33--59, 2007.

\bibitem{B08}
Krzysztof Bara\'nski.
\newblock Hausdorff dimension of hairs and ends for entire maps of finite
  order.
\newblock {\em Math. Proc. Camb. Phil. Soc.}, 145:719--737, 2008.

\bibitem{BKZ-Bowen}
Krzysztof Bara\'nski, Boguslawa Karpi\'{n}ska, and Anna Zdunik.
\newblock Bowen's formula for meromorphic functions.
\newblock {\em Ergodic Theory Dynam. Systems}, 32(4):1165--1189, 2012.

\bibitem{BKZ18}
Krzysztof Bara\'nski, Boguslawa Karpi\'nska, and Anna Zdunik.
\newblock Conformal measures for meromorphic maps.
\newblock {\em Ann. Acad. Sci. Fenn. Math.}, 43(1):247--266, 2018.

\bibitem{BKZ-2009}
Krzysztof Bara{{\'n}}ski,  Boguslawa  Karpi{{\'n}}ska, and Anna
  Zdunik.
\newblock Hyperbolic dimension of {J}ulia sets of meromorphic maps with
  logarithmic tracts.
\newblock {\em Int. Math. Res. Not. IMRN}, (4):615--624, 2009.

\bibitem{Bergweiler-survey}
Walter Bergweiler.
\newblock Iteration of meromorphic functions.
\newblock {\em Bull. Amer. Math., Soc}, 29:151--188, 1993.

\bibitem{BE08.1}
Walter Bergweiler and Alexandre Eremenko.
\newblock Direct singularities and completely invariant domains of entire
  functions.
\newblock {\em Illinois J. Math.}, 52(1):243--259, 2008.

\bibitem{BeurlingAhlfors1956}
Arne Beurling and Lars ~ V. Ahlfors.
\newblock The boundary correspondence under quasiconformal mappings.
\newblock {\em Acta Math.}, 96:125--142, 1956.

\bibitem{Bishop06}
Christopher~J. Bishop.
\newblock A criterion for the failure of {R}uelle's property.
\newblock {\em Ergodic Theory Dynam. Systems}, 26(6):1733--1748, 2006.

\bibitem{Bishop-EL-2015}
Christopher~J. Bishop.
\newblock Models for the {E}remenko-{L}yubich class.
\newblock {\em J. Lond. Math. Soc. (2)}, 92(1):202--221, 2015.

\bibitem{Bishop-S-2016}
Christopher~J. Bishop.
\newblock Models for the Speiser class.
\newblock {\em Proc. Lond. Math. Soc.}, (3) 114(5):765--797, 2017.

\bibitem{Bow79}
Rufus Bowen.
\newblock Hausdorff dimension of quasicircles.
\newblock {\em Inst. Hautes {\'E}tudes Sci. Publ. Math.}, (50):11--25, 1979.

\bibitem{DSZ97}
A. Douady, P. Sentenac and M. Zinsmeister.
\newblock{\em Implosion parabolique et dimension de Hausdorff.}
\newblock CR Acad. Sci. I 325 (1997)

\bibitem{QC and Analysis}
Duren, P., Heinonen, J., Osgood, B., Palka, B. (Eds.).
\newblock Quasiconformal Mappings and Analysis. 
\newblock A Collection of Papers Honoring F.W. Gehring. Springer-Verlag New York, 1998.


\bibitem{EL92}
Alexandre~Eremenko and Mikhail~Yu Lyubich.
\newblock Dynamical properties of some classes of entire functions.
\newblock {\em Annales de l'institut Fourier}, 42(4):989--1020, 1992.

\bibitem{Falconer}
Kenneth Falconer
\newblock Fractal Geometry: Mathematical Foundations and Applications.
\newblock {\em Wiley,} 3rd Edition, 398 pages, 2014.


\bibitem{GO79}
Gehring, F. W. and Osgood, B. G.
\newblock Uniform domains and the quasihyperbolic metric
\newblock {\em J. Analyse Math.} , (36) 50--74, 1979.

\bibitem{HuoWu15}
Shengjin Huo and Shengjian Wu.
\newblock The failure of analyticity of {H}ausdorff dimensions of quasi-circles
  of {F}uchsian groups of the second kind.
\newblock {\em Proc. Amer. Math. Soc.}, 143(3):1101--1108, 2015.


\bibitem{KoUrb}
Janina Kotus and Mariusz Urba\'nski.
\newblock The dynamics and geometry of the Fatou functions.
\newblock Discrete and Continuous Dynamical Systems - A,13,2,291,338,2005-4-1.



\bibitem{LV-1973}
Olli~Lehto and Kalle~I. Virtanen.
\newblock {\em Quasiconformal mappings in the plane}.
\newblock Springer-Verlag, New York-Heidelberg, second edition, 1973.
\newblock Translated from the German by K. W. Lucas, Die Grundlehren der
  mathematischen Wissenschaften, Band 126.

\bibitem{MSS83}
Ricardo ~Ma{\~n}{\'e}, Paulo ~Sad, and Dennis~Sullivan.
\newblock On the dynamics of rational maps.
\newblock {\em Ann. Sci. \'Ecole Norm. Sup. (4)}, 16(2):193--217, 1983.

\bibitem{Martin-1997}
Gaven~J. Martin.
\newblock The distortion theorem for quasiconformal mappings, schottky's
  theorem and holomorphic motions.
\newblock {\em Proc. Amer. Math. Soc.}, 125(4):1095--1103, 1997.




\bibitem{MauldinUrb02}
R.~Daniel Mauldin and Mariusz Urba{{\'n}}ski.
\newblock Dimensions, measures in infinite iterated function systems.
\newblock {\em Proc. London Math. Soc.}, 73(1):105--154, 1996.

\bibitem{Mayer:2017ab}
Volker Mayer.
\newblock A lower bound of the hyperbolic dimension for meromorphic functions
  having a logarithmic h{\"o}lder tract.
\newblock {\em Conformal Geometry and Dynamics}, 22:62--77, 2018.

\bibitem{MyUrb08}
Volker Mayer and Mariusz Urba{{\'n}}ski.
\newblock Geometric thermodynamic formalism and real analyticity for
  meromorphic functions of finite order.
\newblock {\em Ergodic Theory Dynam. Systems}, 28(3):915--946, 2008.

\bibitem{MUmemoirs}
Volker Mayer and Mariusz Urba{{\'n}}ski.
\newblock Thermodynamical formalism and multifractal analysis for meromorphic
  functions of finite order.
\newblock {\em Mem. Amer. Math. Soc.}, 203(954):vi+107, 2010.

\bibitem{MUpreprint3}
Volker Mayer and Mariusz Urbanski.
\newblock Thermodynamic formalism and integral means spectrum of logarithmic
  tracts for transcendental entire functions.
\newblock {\em Trans. Amer. Math. Soc.} 373:7669-7711, 2020. 

\bibitem{Mayer:2016aa}
Volker Mayer, Mariusz Urbanski, and Anna Zdunik.
\newblock Real analyticity for random dynamics of transcendental functions.
\newblock {\em Ergodic Theory Dynamical Systems}, (to appear).

\bibitem{McM87}
Curt McMullen.
\newblock Area and {H}ausdorff dimension of {J}ulia sets of entire functions.
\newblock {\em Trans. Amer. Math. Soc.}, 300(1):329--342, 1987.

\bibitem{Po15}
M.~Pollicott.
\newblock Analyticity of dimensions for hyperbolic surface diffeomorphisms.
\newblock {\em Proc. Amer. Math. Soc.}, 143(8):3465--3474, 2015.

\bibitem{PommerenkeBook}
Christian  Pommerenke.
\newblock {\em Boundary behaviour of conformal maps}, volume 299 of {\em
  Grundlehren der Mathematischen Wissenschaften [Fundamental Principles of
  Mathematical Sciences]}.
\newblock Springer-Verlag, Berlin, 1992.

\bibitem{Rempe09-hyp}
Lasse Rempe.
\newblock Hyperbolic dimension and radial julia sets of transcendental
  functions.
\newblock {\em Proceedings of the American Mathematical Society},
  137(4):1411--1420, 2009.

\bibitem{Rempe09}
Lasse Rempe.
\newblock Rigidity of escaping dynamics for transcendental entire functions.
\newblock {\em Acta Math.}, 203(2):235--267, 2009.

\bibitem{Rempe-HypDim2}
Lasse Rempe-Gillen.
\newblock Hyperbolic entire functions with full hyperbolic dimension and
  approximation by Eremenko-Lyubich functions.
\newblock {\em Proc. Lond. Math. Soc. (3)}, 108(5):1193--1225, 2014.

\bibitem{RempeSixsmith16}
Lasse Rempe-Gillen and Dave Sixsmith.
\newblock Hyperbolic entire functions and the Eremenko--Lyubich class: Class
  $\mathcal B $ b or not class $\mathcal B$ ?
\newblock {\em Mathematische Zeitschrift},  286(3-4): 783--800, 2016.

\bibitem{Rohde-2001}
Steffen.~Rohde.
\newblock Quasicircles modulo bilipschitz maps.
\newblock {\em Rev. Mat. Iberoamericana}, 17:643--659, 2001.

\bibitem{Ru82}
David Ruelle.
\newblock Repellers for real analytic maps.
\newblock {\em Ergodic Theory Dynamical Systems}, 2(1):99--107, 1982.


\bibitem{Rugh}
Hans Henrik Rugh
\newblock On the dimensions of conformal repellers. Randomness and parameter dependency.
\newblock {\em Ann. of Math}, 168-3:695-748, 2008.



\bibitem{Shi98}
M. Shishikura.
\newblock {\em The Hausdorff dimension of the boundary of the Mandelbrot set and Julia sets.}
\newblock Ann. of Math. 147, 225--267, 1998.



\bibitem{SU14}
Bartlomiej Skorulski and Mariusz Urbanski.
\newblock Finer fractal geometry for analytic families of conformal dynamical
  systems.
\newblock {\em Dyn. Syst.}, 29(3):369--398, 2014.


\bibitem{St99-1}
Gwyneth~M. Stallard.
\newblock The {H}ausdorff dimension of {J}ulia sets of hyperbolic meromorphic
  functions.
\newblock {\em Math. Proc. Cambridge Philos. Soc.}, 127(2):271--288, 1999.

\bibitem{Sumi09}
Hiroki Sumi and Mariusz Urba{{\'n}}ski.
\newblock Real analyticity of {H}ausdorff dimension for expanding rational
  semigroups.
\newblock {\em Ergodic Theory Dynam. Systems}, 30(2):601--633, 2010.




\bibitem{UZ03}
Mariusz Urba{{\'n}}ski and Anna Zdunik.
\newblock The finer geometry and dynamics of the hyperbolic exponential family.
\newblock {\em Michigan Math. J.}, 51(2):227--250, 2003.

\bibitem{UZ04}
Mariusz Urba{{\'n}}ski and Anna Zdunik.
\newblock Real analyticity of {H}ausdorff dimension of finer {J}ulia sets of
  exponential family.
\newblock {\em Ergodic Theory Dynam. Systems}, 24(1):279--315, 2004.

\bibitem{VW96}
A.~Verjovsky and H.~Wu.
\newblock Hausdorff dimension of {J}ulia sets of complex {H}\'{e}non mappings.
\newblock {\em Ergodic Theory Dynam. Systems}, 16(4):849--861, 1996.

\bibitem{Zins2000}
Michel Zinsmeister.
\newblock {\em Thermodynamic formalism and holomorphic dynamical systems},
  volume~2 of {\em SMF/AMS Texts and Monographs}.
\newblock American Mathematical Society, Providence, RI; Soci{\'e}t{\'e}
  Math{\'e}matique de France, Paris, 2000.
\newblock Translated from the 1996 French original by C. Greg Anderson.

\end{thebibliography}



\end{document}